\documentclass[12pt,reqno]{amsart}

\usepackage{mathtools}

\usepackage{mathdots}

\usepackage{color}

\usepackage{pdfsync}

\usepackage{enumitem}

\usepackage{wasysym}

\usepackage{amssymb}

\usepackage{mathrsfs}

\DeclareMathAlphabet{\mathpzc}{OT1}{pzc}{m}{it}

\usepackage[all]{xy}

\usepackage{tikz}
\usetikzlibrary{arrows,matrix,decorations.pathmorphing,decorations.pathreplacing,positioning,shapes.geometric,shapes.misc,decorations.markings,decorations.fractals,calc,patterns}

\usepackage{graphicx}

\usepackage{float}

\usepackage[bottom]{footmisc}

\usepackage{moreenum}

\usepackage{makecell}

\usepackage[T1]{fontenc}

\entrymodifiers={+!!<0pt,\fontdimen22\textfont2>}

\def\cA{\mathscr{A}}
\def\cB{\mathscr{B}}
\def\cC{\mathscr{C}}
\def\cD{\mathscr{D}}
\def\cE{\mathscr{E}}

\def\cI{\mathscr{I}}

\def\cM{\mathscr{M}}

\def\cP{\mathscr{P}}

\def\cW{\mathscr{W}}
\def\cX{\mathscr{X}}
\def\cY{\mathscr{Y}}

\def\BZ{\mathbb{Z}}

\def\fa{\mathfrak{a}}

\def\fr{\mathfrak{r}}

\def\Ab{\cA\!{\it b}}

\def\adots{\mathinner{\mkern1mu\raise1.0pt\vbox{\kern7.0pt\hbox{.}}\mkern2mu\raise4.0pt\hbox{.}\mkern2mu\raise7.0pt\hbox{.}\mkern1mu}}

\def\cof{\operatorname{cof}}

\def\Coker{\operatorname{Coker}}

\def\dddots{\mathinner{\mkern1mu\raise10.0pt\vbox{\kern7.0pt\hbox{.}}\mkern2mu\raise5.3pt\hbox{.}\mkern2mu\raise1.0pt\hbox{.}\mkern1mu}}
\def\dddotssmall{\mathinner{\mkern1mu\raise7.0pt\vbox{\kern7.0pt\hbox{.}}\mkern-1mu\raise4pt\hbox{.}\mkern-1mu\raise1.0pt\hbox{.}\mkern1mu}}

\def\dim{\operatorname{dim}}

\def\dual{\operatorname{D}}

\def\Ext{\operatorname{Ext}}

\def\fib{\operatorname{fib}}

\def\Fun{\operatorname{Fun}}

\def\H{\operatorname{H}}

\def\Hom{\operatorname{Hom}}
\def\id{\operatorname{id}}
\def\Image{\operatorname{Im}}

\def\invlim{\underset{ \textstyle \longleftarrow }{ \operatorname{ lim }}}

\def\K{\operatorname{K}}
\def\K0{\operatorname{K}_0}

\def\Ker{\operatorname{Ker}}

\def\lim{\operatorname{lim}}

\def\Mod{\operatorname{Mod}}

\def\obj{{\operatorname{obj}}}
\def\opp{\operatorname{o}}

\def\PSL2{\operatorname{PSL}_2}

\def\rad{\operatorname{rad}}

\def\SL2{\operatorname{SL}_2}

\def\Tor{\operatorname{Tor}}

\def\weq{\operatorname{weq}}

\numberwithin{equation}{section}
\renewcommand{\theequation}{\arabic{section}.\arabic{equation}}

\renewcommand{\thesubsection}{\arabic{section}.\roman{subsection}}

\newtheorem{Lemma}{Lemma}[section]
\newtheorem{Theorem}[Lemma]{Theorem}
\newtheorem{Proposition}[Lemma]{Proposition}

\theoremstyle{definition}
\newtheorem{Definition}[Lemma]{Definition}
\newtheorem{Setup}[Lemma]{Setup}

\newtheorem{Construction}[Lemma]{Construction}
\newtheorem{Remark}[Lemma]{Remark}

\newtheorem{bfhpg}[Lemma]{}               
\newtheorem*{bfhpg*}{}

\newenvironment{VarDescription}[1]%
  {\begin{list}{}{%
    \settowidth{\labelwidth}{\textbf{#1:}}%
    \setlength{\leftmargin}{\labelwidth}\addtolength{\leftmargin}{\labelsep}}}%
  {\end{list}}

\newcounter{saveLemma}

\begin{document}

\setlength{\parindent}{0pt}
\setlength{\parskip}{7pt}

\title[Model categories of quiver representations]{Model categories of quiver representations}

\author{Henrik Holm}

\address{Department of Mathematical Sciences, Universitetsparken 5, University of Copenhagen, 2100 Copenhagen {\O}, Denmark} 
\email{holm@math.ku.dk}

\urladdr{http://www.math.ku.dk/\~{}holm/}

\author{Peter J\o rgensen}

\address{School of Mathematics and Statistics,
Newcastle University, Newcastle upon Tyne NE1 7RU, United Kingdom}
\email{peter.jorgensen@ncl.ac.uk}

\urladdr{http://www.staff.ncl.ac.uk/peter.jorgensen}





\begin{abstract} 

\medskip
\noindent
Gillespie's Theorem gives a systematic way to construct model category structures on $\mathscr{C}( \mathscr{M} )$, the category of chain complexes over an abelian category $\mathscr{M}$.

\medskip
\noindent
We can view $\mathscr{C}( \mathscr{M} )$ as the category of representations of the quiver $\cdots \rightarrow 2 \rightarrow 1 \rightarrow 0 \rightarrow -1 \rightarrow -2 \rightarrow \cdots$ with the relations that two consecutive arrows compose to $0$.  This is a self-injective quiver with relations, and we generalise Gillespie's Theorem to other such quivers with relations.  There is a large family of these, and following Iyama and Minamoto, their representations can be viewed as generalised chain complexes.  

\medskip
\noindent
Our result gives a systematic way to construct model category structures on many categories.  This includes the category of $N$-periodic chain complexes, the category of $N$-complexes where $\partial^N = 0$, and the category of representations of the repetitive quiver $\mathbb{Z} A_n$ with mesh relations.

\end{abstract}

\maketitle

\setcounter{section}{-1}
\section{Introduction}
\label{sec:introduction}

Gillespie's Theorem permits the construction of model category structures on categories of chain complexes.  We will generalise it to representations of self-injective quivers with relations, which can be viewed as generalised chain complexes by the work of Iyama and Minamoto, see \cite{IM1} and \cite[sec.\ 2]{IM2}.

\smallskip
\subsection{Outline}~\\
\label{subsec:outline}
\vspace{-3.4ex}

Let $\cM$ be an abelian category.  An abelian model category structure on $\cC( \cM )$, the category of chain complexes over $\cM$, consists of three classes of morphisms, $( \fib,\cof,\weq )$, known as fibrations, cofibrations, and weak equivalences, subject to several axioms, see \cite[def.\ 2.1]{H} and \cite[sec.\ I.1]{Q}.  It provides an extensive framework for the construction and manipulation of the localisation $\weq^{-1}\cC( \cM )$, where the morphisms in $\weq$ have been inverted formally.  Some of the localisations thus obtained are of considerable interest, not least the derived category $\cD( \cM )$. 

Hovey's Theorem says that each abelian model category structure can be constructed from two so-called complete, compatible cotorsion pairs, see Theorem \ref{thm:Hovey}.  This motivates Gillespie's Theorem, which takes a hereditary cotorsion pair in $\cM$ and produces two compatible cotorsion pairs in $\cC( \cM )$, see Theorem \ref{thm:Gillespie}.    

Gillespie's Theorem can be viewed as a result on quiver representations since $\cC( \cM )$ is the category of representations of $Q$ with values in $\cM$, where $Q$ is the following self-injective  quiver with relations.
\begin{equation}
\label{equ:Gillespie_quiver}
  \left\{
    \begin{array}{cc}
      \mbox{Quiver:}    & \cdots \xrightarrow{} 2 \xrightarrow{} 1 \xrightarrow{} 0 \xrightarrow{} -1 \xrightarrow{} -2  \xrightarrow{} \cdots \\[1mm]
      \mbox{Relations:} & \mbox{Two consecutive arrows compose to $0$.}
    \end{array}
  \right.
\end{equation}
The notion of self-injectivity is made precise in Paragraph \ref{bfhpg:quivers}.  This paper will generalise Gillespie's Theorem to other self-injective quivers with relations. They form a large family, see for example Equations \eqref{equ:intro_quiver_2} and \eqref{equ:intro_quiver_4} and Section \ref{subsec:other_quivers}.

Let $k$ be a field, $R$ a $k$-algebra, $Q$ a self-injective quiver with relations over $k$, and let $\cX$ be the category of representations of $Q$ with values in ${}_{ R }\!\Mod$, the category of $R$-left-modules.  Our main theorem,  Theorem \ref{thm:main}, takes a hereditary cotorsion pair in ${}_{ R }\!\Mod$ and produces two compatible cotorsion pairs in $\cX$.  It specialises to Gillespie's Theorem for $\cM = {}_{ R }\!\Mod$ if $Q$ is the quiver with relations from \eqref{equ:Gillespie_quiver}.

\bigskip
\subsection{Cotorsion pairs}~\\
\vspace{-2.5ex}

Let $\cY$ be an abelian category.  If $\Gamma$ and $\Delta$ are classes of objects of $\cY$, then we write
\[
  \Gamma^{ \perp } = \{\, Y \in \cY \mid \Ext_{ \cY }^1( C,Y ) = 0 \mbox{ for } C \in \Gamma \,\}
  \;\;,\;\;  
  {}^{ \perp }\Delta = \{\, Y \in \cY \mid \Ext_{ \cY }^1( Y,D ) = 0 \mbox{ for } D \in \Delta \,\}.
\]
\begin{Definition}
\label{def:cotorsion}
Recall the following from the literature.
\begin{enumerate}
\setlength\itemsep{4pt}

  \item  A {\em cotorsion pair} in $\cY$ is a pair $( \Gamma,\Delta )$ of classes of objects of $\cY$ such that $\Gamma = {}^{ \perp }\Delta$ and $\Gamma^{ \perp } = \Delta$, see \cite[p.\ 12]{S}.  A cotorsion pair $( \Gamma,\Delta )$ is determined by each of the classes $\Gamma$ and $\Delta$, because it is equal to $( \Gamma,\Gamma^{ \perp } )$ and to $( {}^{ \perp }\Delta,\Delta )$.

  \item  The cotorsion pair $( \Gamma,\Delta )$ in $\cY$ is {\em complete} if each $Y \in \cY$ permits short exact sequences $0 \xrightarrow{} D \xrightarrow{} C \xrightarrow{} Y \xrightarrow{} 0$ and $0 \xrightarrow{} Y \xrightarrow{} D' \xrightarrow{} C' \xrightarrow{} 0$ with $C,C' \in \Gamma$ and $D,D' \in \Delta$, see \cite[lem.\ 5.20]{GT}.

  \item  The cotorsion pair $( \Gamma,\Delta )$ is {\em hereditary} if $\Gamma$ is closed under kernels of epimorphisms and $\Delta$ is closed under cokernels of monomorphisms, see \cite[lem.\ 5.24]{GT}.
  
  \item  The cotorsion pairs $( \Phi,\Phi^{ \perp } )$ and $( {}^{ \perp }\Psi,\Psi )$ in $\cY$ are {\em compatible} if they satisfy the following conditions, see \cite[sec.\ 1]{G1}.
\medskip
\begin{VarDescription}{Comp2\quad}
\setlength\itemsep{5pt}

  \item[(Comp1)]  $\Ext^1_{ \cY }( \Phi,\Psi ) = 0$.

  \item[(Comp2)]  $\Phi \cap \Phi^{ \perp } = {}^{ \perp }\Psi \cap \Psi$.
  
\end{VarDescription}
\medskip
Condition (Comp1) is equivalent to $\Phi \subseteq {}^{ \perp }\Psi$ and to $\Phi^{ \perp } \supseteq \Psi$.  It is not symmetric in the two cotorsion pairs; their order matters.  

\medskip
\noindent
Note that our definition of compatibility is weaker than Gillespie's from \cite[def.\ 3.7]{G2}, and that his cortorsion pairs $( \tilde{ \cA },\operatorname{dg} \tilde{ \cB } )$ and $( \operatorname{dg} \tilde{ \cA },\tilde{ \cB } )$ in $\cC( \cM )$ from \cite[prop.\ 3.6]{G2} are always compatible in our sense.  Indeed, $\tilde{ \cA } \cap \operatorname{dg} \tilde{ \cB }$ and $\operatorname{dg} \tilde{ \cA } \cap \tilde{ \cB }$ are both equal to the class of split exact complexes with terms in $\cA \cap \cB$.

  \item  Let $( \Gamma,\Delta )$ be a cotorsion pair in $\cY$, and let $\cC$ be a class of objects in $\cY$.  If $\Delta = \cC^{ \perp }$, then we say that $( \Gamma,\Delta )$ is {\em generated} by $\cC$.  If $\Gamma = {}^{ \perp }\cC$, then we say that $( \Gamma,\Delta )$ is {\em cogenerated} by $\cC$.  See \cite[def.\ 5.15]{GT}.

\end{enumerate}
\end{Definition}

For example, if $\cY$ has enough projective objects, then $( \mbox{projective objects},\cY )$ is called the {\em projective cotorsion pair}.  If $\cY$ has enough injective objects, then $( \cY,\mbox{injective objects} )$ is called the {\em injective cotorsion pair}.  These cotorsion pairs are complete and hereditary.  Note that the triangulated version of compatible cotorsion pairs was investigated by Nakaoka under the name concentric twin cotorsion pair, see \cite[def.\ 3.3]{N}.

\bigskip
\subsection{Hovey's Theorem: Abelian model category structures}~\\
\vspace{-2.5ex}

We will not reproduce Hovey's Theorem in full, but rather state the following result, which motivates the interest in compatible cotorsion pairs and dovetails with Gillespie's Theorem.

\begin{Theorem}[{\cite[prop.\ 2.3 and sec.\ 4.2]{G0}, \cite[thm.\ 1.1]{G1}, \cite[thm.\ 2.2]{H}}]
\label{thm:Hovey}
Let $( \Phi,\Phi^{ \perp } )$ and $( {}^{ \perp }\Psi,\Psi )$ be complete, hereditary, compatible cotorsion pairs in the abelian category $\cY$.

There is a class $\cW$ of objects, often referred to as {\em trivial}, characterised by
\begin{align*}
  \cW & = \{\, Y \in \cY \mid \mbox{there is a short exact sequence $0 \xrightarrow{} P \xrightarrow{} F \xrightarrow{} Y \xrightarrow{} 0$ with $P \in \Psi$, $F \in \Phi$} \,\} \\[1mm]
  & = \{\, Y \in \cY \mid \mbox{there is a short exact sequence $0 \xrightarrow{} Y \xrightarrow{} P' \xrightarrow{} F' \xrightarrow{} 0$ with $P' \in \Psi$, $F' \in \Phi$} \,\}.
\end{align*}
Moreover, there is a model category structure on $\cY$ with
\begin{align*}
  \fib
  & =
  \{\,\mbox{epimorphisms with kernel in $\Phi^{ \perp }$}\,\}, \\[1mm]
  \cof
  & =
  \{\,\mbox{monomorphisms with cokernel in ${}^{ \perp }\Psi$}\,\}, \\[1mm]
  \weq
  & =
  \left\{
    \begin{array}{l}
      \mbox{morphisms which factor as a monomorphism with cokernel in $\cW$} \\
      \mbox{followed by an epimorphism with kernel in $\cW$}    	
    \end{array}
  \right\},
\end{align*}
and the localisation $\weq^{ -1 }\cY$ is triangulated.
\end{Theorem}

We will give an example after recalling Gillespie's Theorem.

\bigskip
\subsection{Gillespie's Theorem: Chain complexes}~\\
\label{subsec:Gillespie}
\vspace{-2.5ex}

Gillespie's Theorem gives a systematic way to construct compatible cotorsion pairs in the category of chain complexes.  It requires the following setup.
\begin{itemize}
\setlength\itemsep{4pt}

  \item  $\cM$ is an abelian category with enough projective and enough injective objects.

  \item  $\cC( \cM )$ is the category of chain complexes over $\cM$.

  \item  For $q \in \BZ$, consider the functors
\[
\vcenter{
\xymatrix
{
  \cM
    \ar[rrr]^{ S_q } &&&
  \cC( \cM ),
    \ar@/_2.0pc/[lll]_{ C_q }
    \ar@/^2.0pc/[lll]^{ K_q }
}
        }
\]
where $S_q$ sends $M$ to the chain complex with $M$ in degree $q$ and zero everywhere else, and  $C_q$ and $K_q$ are given by
\[
  C_q( X ) = \Coker( \partial_{ q+1 }^X )
  \;\;,\;\;
  K_q( X ) = \Ker( \partial_q^X ).
\]
Here $\partial_q^X$ is the $q$th differential of the chain complex $X$.  There are adjoint pairs $( C_q,S_q )$ and $( S_q,K_q )$.

\end{itemize}

The following is Gillespie's Theorem.

\begin{Theorem}[{\cite[thm.\ 3.12 and cor.\ 3.13]{G2}}]
\label{thm:Gillespie}
If $( \cA,\cB )$ is a hereditary cotorsion pair in $\cM$, then there are hereditary, compatible cotorsion pairs $\big( \Phi( \cA ),\Phi( \cA )^{ \perp } \big)$ and $\big( {}^{ \perp }\Psi( \cB ),\Psi( \cB ) \big)$ 
in $\cC( \cM )$, where
\begin{align*}
  \Phi( \cA ) & = \{\, X \in \cC( \cM ) \mid \mbox{If $q \in \BZ$ then $C_q( X ) \in \cA$ and $\H_q( X ) = 0$} \,\}, \\[1mm]
  \Psi( \cB ) & = \{\, X \in \cC( \cM ) \mid \mbox{If $q \in \BZ$ then $K_q( X ) \in \cB$ and $\H_q( X ) = 0$} \,\}.
\end{align*}
\end{Theorem}

For instance, the projective cotorsion pair $( \cA,\cB ) = (\mbox{projective objects},\cM )$ gives
\begin{equation}
\label{equ:intro11}
    \big( \Phi( \cA ),\Phi( \cA )^{ \perp } \big) = ( \cP,\cP^{ \perp } )
  \;\;,\;\;  
  \big( {}^{ \perp }\Psi( \cB ),\Psi( \cB ) \big) = ( {}^{ \perp }\cE,\cE ),
\end{equation}
where $\cP$ is the class of projective objects in $\cC( \cM )$ and $\cE$ is the class of exact chain complexes.  Note that $\cP^{ \perp } = \cC( \cM )$.  The cotorsion pairs \eqref{equ:intro11} are hereditary and compatible by Gillespie's Theorem.  If $\cM$ is a complete and cocomplete category, then the cotorsion pairs \eqref{equ:intro11} are complete, and then Theorem \ref{thm:Hovey} says that they determine an abelian model category structure on $\cC( \cM )$.  The associated localisation $\weq^{ -1 }\cC( \cM )$ is the derived category $\cD( \cM )$, see \cite[thm.\ 5.3]{G0}.

\bigskip
\subsection{The Main Theorem: Quiver representations}~\\
\label{subsec:Main}
\vspace{-2.5ex}

Our main theorem is a generalisation of Gillespie's Theorem to quiver representations.  It requires the following setup, which we keep in the rest of the introduction.
\begin{itemize}
\setlength\itemsep{4pt}

  \item  $k$ is a field, $R$ is a $k$-algebra, ${}_{ R }\!\Mod$ is the category of $R$-left-modules.
  
  \item  $Q$ is a self-injective quiver with relations over $k$, see Paragraph \ref{bfhpg:quivers}.

  \item  $\cX$ is the category of representations of $Q$ with values in ${}_{ R }\!\Mod$.  If $p \xrightarrow{ \pi } q$ is an arrow in $Q$, then the corresponding homomorphism in $X \in \cX$ is $X_p \xrightarrow{ X_{ \pi } } X_q$.

  \item  For $q$ an element of $Q_0$, the set of  vertices of $Q$, consider the functors
\[
\vcenter{
\xymatrix
{
  {}_{ R }\!\Mod
    \ar[rrr]^{ S_q } &&&
  \cX
    \ar@/_2.0pc/[lll]_{ C_q }
    \ar@/^2.0pc/[lll]^{ K_q }
}
        }
\]
defined by:
\[
  C_q( - ) = \dual\!S \langle q \rangle \underset{ Q }{ \otimes } -
  \;\;\;,\;\;\;
  S_q( - ) = S \langle q \rangle \underset{ k }{ \otimes } -
  \;\;\;,\;\;\;
  K_q( - ) = \Hom_Q( S \langle q \rangle,- ).
\]
Here $S \langle q \rangle$ is the simple representation of $Q$
supported at $q$.  Its dual $\dual\!S \langle q \rangle = \Hom_k( S \langle q \rangle,k )$ is the simple representation of the opposite quiver $Q^{ \opp }$ supported at $q$.  The symbols $\underset{ Q }{ \otimes }$ and $\Hom_Q$ denote the tensor product and homomorphism functors of representations of $Q$.  Note that $S_q$ sends $M$ to the representation with $M$ at vertex $q$ and zero everywhere else.  There are adjoint pairs $( C_q,S_q )$ and $( S_q,K_q )$.

\end{itemize}

Our main theorem is the following.

\begingroup
\setcounter{saveLemma}{\value{Lemma}}
\setcounter{Lemma}{0}
\renewcommand{\theLemma}{\Alph{Lemma}}
\begin{Theorem}
\label{thm:main}
If $( \cA,\cB )$ is a hereditary cotorsion pair in ${}_{ R }\!\Mod$, then there are hereditary, compatible cotorsion pairs $\big( \Phi( \cA ),\Phi( \cA )^{ \perp } \big)$ and $\big( {}^{ \perp }\Psi( \cB ),\Psi( \cB ) \big)$ in $\cX$, the category of representations of $Q$ with values in ${}_{ R }\!\Mod$, where
\begin{align*}
  \Phi( \cA ) & = \{\, X \in \cX \mid \mbox{If $q \in Q_0$ then $C_q( X ) \in \cA$ and $L_1C_q( X ) = 0$} \,\}, \\[1mm]
  \Psi( \cB ) & = \{\, X \in \cX \mid \mbox{If $q \in Q_0$ then $K_q( X ) \in \cB$ and $R^1K_q( X ) = 0$} \,\}.
\end{align*}
\end{Theorem}
\setcounter{Lemma}{\value{saveLemma}}
\endgroup

In the body of the paper, we prove the more general Theorem \ref{thm:self-injective_main} where $Q$ is a small $k$-preadditive category, and $\cX$ is the functor category of $k$-linear functors $Q \rightarrow {}_{ R }\!\Mod$.  Paragraph \ref{bfhpg:quivers} explains how a quiver can be viewed as a category, whence Theorem \ref{thm:self-injective_main} specialises to Theorem \ref{thm:main}.  

Theorem \ref{thm:main} specialises to Gillespie's Theorem for $\cM = {}_{ R }\!\Mod$ if $Q$ is the quiver with relations from \eqref{equ:Gillespie_quiver}.  Then $\cX$ is the category of chain complexes over ${}_{ R }\!\Mod$.  A computation shows that the functors $C_q$, $S_q$, $K_q$ specialise to those of Section \ref{subsec:Gillespie}, and that
\[
  L_1 C_q = \H_{ q+1 }
  \;\;,\;\;
  R^1 K_q = \H_{ q-1 },
\]
whence the formulae in Theorem \ref{thm:main} specialise to those in Gillespie's Theorem.  However, Theorem \ref{thm:main} applies to many other quivers with relations.  Following Iyama and Minamoto \cite[def.\ 8]{IM2}, we then think of $L_1 C_q$ and $R^1 K_q$ as generalised homology functors.

To serve as the input for Theorem \ref{thm:Hovey}, the cotorsion pairs $\big( \Phi( \cA ),\Phi( \cA )^{ \perp } \big)$ and $\big( {}^{ \perp }\Psi( \cB ),\Psi( \cB ) \big)$ must be complete.  In the setup of Theorem \ref{thm:Gillespie}, this is indeed true under the conditions that $\cM$ is a complete and cocomplete category and $( \cA,\cB )$ is a complete cotorsion pair, see \cite[thm.\ 2.4]{DY}.  In the more complicated setup of Theorem \ref{thm:main}, we do not have an equally neat result, but we do prove completeness in certain cases, see Theorem \ref{thm:self-injective_complete}.

\bigskip
\subsection{Application: $N$-periodic chain complexes}~\\
\vspace{-2.5ex}
\label{subsec:main2}

Let $N \geqslant 1$ be an integer.  
\begin{itemize}
\setlength\itemsep{4pt}

  \item  In Section \ref{subsec:main2} only, $Q$ is the following quiver with relations. 
\begin{equation}
\label{equ:intro_quiver_2}
  \left\{
    \begin{array}{cc}
      \mbox{Quiver:}    &
\xymatrix
{
  N-1 \ar[r] & N-2 \ar[r] & \cdots \ar[r] & 1 \ar[r] & 0 \ar@/^1.5pc/[llll] \\
}
\\[6mm]
      \mbox{Relations:} & \mbox{Two consecutive arrows compose to $0$}
    \end{array}
  \right.
\end{equation}
This is a self-injective quiver with relations, see Paragraph \ref{bfhpg:quivers}.

\end{itemize}
An object $X \in \cX$ has the form
\[
\xymatrix
{
  X_{ N-1 } \ar^-{ \partial^X_{ N-1 } }[r] & X_{ N-2 } \ar^-{ \partial^X_{ N-2 } }[r] & \cdots \ar[r] & X_1 \ar^-{ \partial^X_1 }[r] & X_0, \ar^{ \partial^X_0 }@/^1.5pc/[llll] \\
}
\]
where two consecutive morphisms compose to $0$.  Hence $\cX$ is the category of $N$-periodic chain complexes over ${}_{ R }\!\Mod$.  This even makes sense for $N = 1$, in which case $X$ is a so-called module with differentiation in the sense of \cite[sec.\ IV.1]{CE}, consisting of an object $X_0 \in {}_{ R }\!\Mod$ and a morphism $X_0 \xrightarrow{ \partial^X_0 } X_0$ squaring to $0$.

For $0 \leqslant q \leqslant N-1$ there is a homology functor $\cX \xrightarrow{ \H_q } {}_{ R }\!\Mod$ defined in an obvious fashion.  
We will use our theory to prove the following.

\begingroup
\setcounter{saveLemma}{\value{Lemma}}
\setcounter{Lemma}{1}
\renewcommand{\theLemma}{\Alph{Lemma}}
\begin{Theorem}
\label{thm:main2}
Let $( \cA,\cB )$ be a hereditary cotorsion pair in ${}_{ R }\!\Mod$.
\begin{enumerate}

	\item  There are hereditary, compatible cotorsion pairs $\big( \Phi( \cA ),\Phi( \cA )^{ \perp } \big)$ and $\big( {}^{ \perp }\Psi( \cB ),\Psi( \cB ) \big)$ in $\cX$, the category of $N$-periodic chain complexes over ${}_{ R }\!\Mod$, where
\begin{align*}
  \Phi( \cA ) & = \{\, X \in \cX \mid \mbox{If $0 \leqslant q \leqslant N-1$ then $\Coker( \partial_q^X ) \in \cA$ and $\H_q( X ) = 0$} \,\}, \\[1mm]
  \Psi( \cB ) & = \{\, X \in \cX \mid \mbox{If $0 \leqslant q \leqslant N-1$ then $\Ker( \partial_q^X ) \in \cB$ and $\H_q( X ) = 0$} \,\}.
\end{align*}

  \item  If $\cA$ is closed under pure quotients and $( \cA,\cB )$ is generated by a set, then the cotorsion pairs in part (i) are complete.

\end{enumerate}
\end{Theorem}
\setcounter{Lemma}{\value{saveLemma}}
\endgroup

This applies to the so-called flat cotorsion pair $( \cA,\cB ) = ( \mbox{flat modules},\mbox{cotorsion modules} )$: Heredity holds by \cite[thm.\ 8.1(a)]{GT}, the class of flat modules is easily seen to be closed under pure quotients, and generation by a set holds by \cite[prop.\ 2]{BBE} (in which ``cogenerated'' means the same as our ``generated'').  Hence Theorem \ref{thm:main2} provides an $N$-periodic version of Gillespie's result for chain complexes from \cite{G2} (see theorem 3.12 and corollaries 3.13, 4.10, 4.18 in that paper).  Theorem \ref{thm:main2} also applies to the injective cotorsion pair $( \cA,\cB ) = ( {}_{ R }\!\Mod,\mbox{injective modules} )$.

\bigskip
\subsection{Application: $\BZ A_3$ with mesh relations}~\\
\vspace{-2.5ex}
\label{subsec:ZA_3}

The following is a slightly more complicated example.
\begin{itemize}
\setlength\itemsep{4pt}

  \item  In Section \ref{subsec:ZA_3} only, $Q$ is the repetitive quiver $\BZ A_3$ modulo the mesh relations.  That is, $Q$ is
\begin{equation}
\label{equ:intro_quiver_4}
  \vcenter{
  \xymatrix @+0.5pc @!0 {
    & (3,2) \ar[dr] && (2,2) \ar[dr] && (1,2) \ar[dr] && (0,2) \ar[dr] && (-1,2) \\
    \cdots && (2,1) \ar[dr] \ar[ur] && (1,1) \ar[dr] \ar[ur] && (0,1) \ar[dr] \ar[ur] && (-1,1) \ar[dr] \ar[ur] && \cdots \\
    & (2,0) \ar[ur] && (1,0) \ar[ur] && (0,0) \ar[ur] && (-1,0) \ar[ur] && (-2,0) \\
                        }
          }
\end{equation}
modulo the relations that each composition of the form $\vcenter{ \xymatrix @-0.9pc @!0 {& *{\circ} \ar[dr] \\ *{\circ} \ar[ur] && *{\circ} } }$ or $\vcenter{ \xymatrix @-0.9pc @!0 {*{\circ} \ar[dr] && *{\circ} \\ & *{\circ} \ar[ur] } }$, which starts and ends on the edge of the quiver, is zero, and that each square of the form $\vcenter{ \xymatrix @-0.9pc @!0 {& *{\circ} \ar[dr] \\ *{\circ} \ar[ur] \ar[dr] && *{\circ} \\ & *{\circ} \ar[ur] } }$ is anticommutative.  This is a self-injective quiver with relations, see Paragraph \ref{bfhpg:quivers}.

\end{itemize}

For $j \in \BZ$, the mesh relations imply that there are short chain complexes
\begin{align}
\nonumber
  & X_{ (j,0) } \xrightarrow{} X_{ (j,1) } \xrightarrow{} X_{ (j-1,0) }, \\[1mm]
\label{equ:ThmBComplexes}
  & X_{ (j,1) } \xrightarrow{} X_{ (j-1,0) } \oplus X_{ (j,2) } \xrightarrow{} X_{ (j-1,1) }, \\[1mm]
\nonumber
  & X_{ (j,2) } \xrightarrow{} X_{ (j-1,1) }\xrightarrow{} X_{ (j-1,2) }.
\end{align}
We will use our theory to prove the following.

\begingroup
\setcounter{saveLemma}{\value{Lemma}}
\setcounter{Lemma}{2}
\renewcommand{\theLemma}{\Alph{Lemma}}
\begin{Theorem}
\label{thm:main3}
If $( \cA,\cB )$ is a hereditary cotorsion pair in ${}_{ R }\!\Mod$, then there are hereditary, compatible cotorsion pairs $\big( \Phi( \cA ),\Phi( \cA )^{ \perp } \big)$ and $\big( {}^{ \perp }\Psi( \cB ),\Psi( \cB ) \big)$ in $\cX$, the category of representations of $Q$ with values in ${}_{ R }\!\Mod$, where
\begin{align*}
  \Phi( \cA ) & =
  \left\{
    \begin{array}{c|c}
      X \in \cX &
      \begin{array}{l}
        \mbox{If $j \in \BZ$ then each of the following cokernels is in $\cA$:} \\[2mm]
        \Coker( X_{ (j,1) } \xrightarrow{} X_{ (j-1,0) }  ), \\[1mm]
        \Coker( X_{ (j,0) } \oplus X_{ (j+1,2) } \xrightarrow{} X_{ (j,1) } ), \\[1mm]
        \Coker( X_{ (j,1) } \xrightarrow{} X_{ (j,2) } ), \\[2mm]
		\mbox{and each of the short chain complexes \eqref{equ:ThmBComplexes} is exact}
      \end{array}
    \end{array}
  \right\}, \\[2mm]
  \Psi( \cB ) & = 
  \left\{
    \begin{array}{c|c}
      X \in \cX &
      \begin{array}{l}
        \mbox{If $j \in \BZ$ then each of the following kernels is in $\cB$:} \\[2mm]
        \Ker( X_{ (j,0) } \xrightarrow{} X_{ (j,1) } ), \\[1mm]
        \Ker( X_{ (j,1) } \xrightarrow{} X_{ (j-1,0) } \oplus X_{ (j,2) } ), \\[1mm]
        \Ker( X_{ (j,2) } \xrightarrow{} X_{ (j-1,1) } ), \\[2mm]
		\mbox{and each of the short chain complexes \eqref{equ:ThmBComplexes} is exact}
      \end{array}
    \end{array}
  \right\}.
\end{align*}
\end{Theorem}
\setcounter{Lemma}{\value{saveLemma}}
\endgroup

\bigskip
\subsection{Other self-injective quivers with relations}~\\
\vspace{-2.5ex}
\label{subsec:other_quivers}

There are many other self-injective quivers with relations to which Theorem \ref{thm:main} can be applied, for instance $\cdots \xrightarrow{} 2 \xrightarrow{} 1 \xrightarrow{} 0 \xrightarrow{} -1 \xrightarrow{} -2  \xrightarrow{} \cdots$ with the relations that $N$ consecutive arrows compose to $0$.  Then $\cX$ is the category of $N$-complexes over ${}_{ R }\!\Mod$ in the sense of \cite[def.\ 0.1]{K}.  Other possibilities are $\BZ A_n$ with mesh relations, the quiver with relations of a finite dimensional self-injective $k$-algebra, and quivers with relations of repetitive algebras, see \cite[sec.\ 3.1]{LP} and \cite[sec.\ 2]{Sch}.

\bigskip
\subsection{An observation on the model category literature}~\\
\vspace{-2.5ex}

Observe that Theorem \ref{thm:main} does not assume the existence of a model category structure on ${}_{ R }\!\Mod$.  This is in contrast to several results from the literature, where a model category structure on a functor category $\Fun( \cI,\cM )$ is induced by a model category structure on $\cM$.  If $\cI$ is a small category, then such results exist when $\cM$ has a cofibrantly generated or combinatorial model category structure, see \cite[thm.\ 11.6.1]{Hirsch} and \cite[prop.\ A.2.8.2]{L}, and when $\cM$ has an arbitrary model category structure and $\cI$ is a direct, an inverse, or a Reedy category, see \cite[thms.\ 5.1.3 and 5.2.5]{HBook}.

\bigskip
\subsection{Contents of the paper}~\\
\vspace{-2.5ex}

Section \ref{sec:Phi_Psi} defines the cotorsion pairs $\big( \Phi( \cA ),\Phi( \cA )^{ \perp } \big)$ and $\big( {}^{ \perp }\Psi( \cB ),\Psi( \cB ) \big)$ in an abstract setup, and shows that they are hereditary and compatible under certain assumptions.  Section \ref{sec:functor_category} introduces functor categories.  Section \ref{sec:self-injective} proves Theorem \ref{thm:self-injective_main}, which has Theorem \ref{thm:main} as a special case.  Sections \ref{sec:Ex}, \ref{sec:Comp1}, and \ref{sec:Seq} provide several results used in the proof of Theorem \ref{thm:self-injective_main}.  Section \ref{sec:main2} proves Theorem \ref{thm:main2}.  Section \ref{sec:ZA_3} proves Theorem \ref{thm:main3}.  Appendix \ref{app:functor} provides additional background on functor categories.

\section{The cotorsion pairs $\big( \Phi( \cA ),\Phi( \cA )^{ \perp } \big)$ and $\big( {}^{ \perp }\Psi( \cB ),\Psi( \cB ) \big)$ in an abstract setup}
\label{sec:Phi_Psi}

This section defines the cotorsion pairs $\big( \Phi( \cA ),\Phi( \cA )^{ \perp } \big)$ and $\big( {}^{ \perp }\Psi( \cB ),\Psi( \cB ) \big)$ in an abstract setup, and shows that they are hereditary and compatible under certain assumptions.  

\begin{Setup}
\label{set:blanket}
Section \ref{sec:Phi_Psi} uses the following setup.
\begin{itemize}
\setlength\itemsep{4pt}

  \item  $\cM$ and $\cX$ are abelian categories with enough projective and enough injective objects.  

  \item  $( \cA,\cB )$ is a cotorsion pair in $\cM$.

  \item  $J$ is an index set.
  
  \item  For each $j \in J$ there are adjoint pairs of functors $( C_j,S_j )$ and $( S_j,K_j )$ as follows.
\[
\vcenter{
\xymatrix
{
\cM
    \ar[rrr]^{ S_j } &&&
  \cX
    \ar@/_2.0pc/[lll]_{ C_j }
    \ar@/^2.0pc/[lll]^{ K_j }
}
        }
\]
Note that this implies that $S_j$ is exact.
\end{itemize}

\end{Setup}

The following lemma provides a so-called ``five term exact sequence''.  It is classic, but we show the proof because we do not have a reference for the precise statement.

\begin{Lemma}
\label{lem:five_term_exact_sequence}
Let $( C,S )$ be an adjoint pair of functors as follows:
$
  \xymatrix {
  \cM
    \ar[r]<-1ex>_-{ S } &
    \cX.
    \ar[l]<-1ex>_-{ C }
            }
$
Assume that $S$ is exact.  For $N \in \cM$ and $X \in \cX$ there is an exact sequence
\[
  0
  \rightarrow \Ext^1_{ \cM }( CX,N )
  \rightarrow \Ext^1_{ \cX }( X,SN )
  \rightarrow \Hom_{ \cM }( L_1CX,N )
  \rightarrow \Ext^2_{ \cM }( CX,N )
  \rightarrow \Ext^2_{ \cX }( X,SN ).
\]
\end{Lemma}

\begin{proof}
Consider the functors $\cX \xrightarrow{ C } \cM \xrightarrow{ B } \Ab$ where $\Ab$ is the category of abelian groups and $B( - ) = \Hom_{ \cM }( -,N )$.

The contravariant functor $B$ is left exact.  If $P \in \cX$ is projective, then $C( P ) \in \cM$ is projective because $\Hom_{ \cM }( CP,- ) \simeq \Hom_{ \cX }\big( P,S(-) \big)$ is an exact functor since $S$ is exact.  In particular, the functor $C$ maps projective objects to right $B$-acyclic objects, that is, objects on which the derived functors $R^{ \geqslant 1 }B$ vanish.

By \cite[thm.\ 10.49]{R} there is a Grothendieck third quadrant spectral sequence
\[
  E_2^{i\ell} = ( R^i B )( L_{\ell} C )X
  \underset{ i }{ \Rightarrow } R^n( BC )X.
\]
If $P$ is a projective resolution of $X$, then
\[
  R^n( BC )X
  \cong \H^n( BCP )
  = \H^n\!\Hom_{ \cM }( CP,N )
  \cong \H^n\!\Hom_{ \cX }( P,SN )
  \cong \Ext^n_{ \cX }( X,SN ).
\]
Hence the spectral sequence is
\[
  E_2^{i\ell} = \Ext^i_{ \cM }( L_{\ell}CX,N )
  \underset{ i }{ \Rightarrow } \Ext^n_{ \cX }( X,SN ).
\]
By \cite[thm.\ 10.33]{R} there is an associated exact sequence, which gives the sequence in the lemma.
\end{proof}

We record the dual without a proof:

\begin{Lemma}
\label{lem:five_term_exact_sequence_dual}
Let $( S,K )$ be an adjoint pair of functors as follows:
$
  \xymatrix {
  \cM
    \ar[r]<1ex>^-{ S } &
    \cX.
    \ar[l]<1ex>^-{ K }
            }
$
Assume that $S$ is exact.  For $N \in \cM$ and $X \in \cX$ there is an exact sequence
\[
  0
  \rightarrow \Ext^1_{ \cM }( N,KX )
  \rightarrow \Ext^1_{ \cX }( SN,X )
  \rightarrow \Hom_{ \cM }( N,R^1KX )
  \rightarrow \Ext^2_{ \cM }( N,KX )
  \rightarrow \Ext^2_{ \cX }( SN,X ).
\]
\end{Lemma}

The following is well known.

\begin{Lemma}
\label{lem:hereditary}
A cotorsion pair $( \cA,\cB )$ in $\cM$ is hereditary if and only if $\cA$ is resolving, that is, contains the projective objects and is closed under kernels of epimorphisms.
\end{Lemma}

\begin{proof}
See \cite[lem.\ 5.24]{GT}, the proof of which works in the present generality.
\end{proof}

\begin{Definition}
\label{def:Phi_Psi}
Let
\begin{align*}
  \cE_L & = \{\, X \in \cX \mid \mbox{If $j \in J$ then $L_1C_j( X ) = 0$} \,\}, \\[1mm]
  \cE_R & = \{\, X \in \cX \mid \mbox{If $j \in J$ then $R^1K_j( X ) = 0$} \,\}.
\end{align*}
If $\cC$ is a class of objects in $\cM$, then let
\begin{align*}
  \Phi( \cC ) & = \{\, X \in \cX \mid 
                       \mbox{If $j \in J$ then $C_j( X ) \in \cC$ and $L_1C_j( X ) = 0$} \,\}, \\[1mm]
  \Psi( \cC ) & = \{\, X \in \cX \mid 
                       \mbox{If $j \in J$ then $K_j( X ) \in \cC$ and $R^1K_j( X ) = 0$} \,\}.
\end{align*}
Note that $\Phi( \cC ) \subseteq \cE_L$ and $\Psi( \cC ) \subseteq \cE_R$.
\end{Definition}

If $\cC$ is a class of objects in $\cM$, then we use the shorthand $\{\, S_*( \cC ) \,\} = \{\, S_j( C ) \mid j \in J, C \in \cC \,\}$.

\begin{Lemma}
\label{lem:perps}
Let $\cC$ be a class of objects in $\cM$.
\begin{enumerate}
\setlength\itemsep{4pt}

  \item  Assume that for each non-zero $M \in \cM$ there is an injective object $I$ which is in $\cC$ and satisfies $\Hom_{ \cM }( M,I ) \neq 0$.  Then $\Phi( {}^{ \perp }\cC ) = {}^{ \perp }\{\, S_*( \cC ) \,\}$.

  \item  Assume that for each non-zero $M \in \cM$ there is a projective object $P$ which is in $\cC$ and satisfies $\Hom_{ \cM }( P,M ) \neq 0$.  Then $\Psi( \cC^{ \perp } ) = \{\, S_*( \cC ) \,\}^{ \perp }$.

\end{enumerate}
\end{Lemma}

\begin{proof}
First note that for $N \in \cC$, $X \in \cX$, $j \in J$, there is an exact sequence
\begin{equation}
\label{equ:five_term_exact_sequence}
  0
  \rightarrow \Ext^1_{ \cM }( C_j X,N )
  \rightarrow \Ext^1_{ \cX }( X,S_j N )
  \rightarrow \Hom_{ \cM }( L_1 C_j X,N )
  \rightarrow \Ext^2_{ \cM }( C_j X,N )
\end{equation}
by Lemma \ref{lem:five_term_exact_sequence}.

Part (i), the inclusion $\subseteq$: Let $X \in \Phi( {}^{ \perp }\cC )$ and $j \in J$ be given.  Then $C_j( X ) \in {}^{ \perp }\cC$ and $L_1 C_j( X ) = 0$ by the definition of $\Phi$.  It follows that for $N \in \cC$, the terms in \eqref{equ:five_term_exact_sequence} which involve $\Ext_{ \cM }^1$ and $\Hom_{ \cM }$ are $0$, so \eqref{equ:five_term_exact_sequence} implies $\Ext^1_{ \cX }( X,S_j N ) = 0$.  Hence $X \in {}^{ \perp }\{\, S_*( \cC ) \,\}$.  

Part (i), the inclusion $\supseteq$:  Let $X \in {}^{ \perp }\{\, S_*( \cC ) \,\}$ and $j \in J$ be given.  

For $N \in \cC$, the term in \eqref{equ:five_term_exact_sequence} which involves $\Ext_{ \cX }^1$ is $0$, so \eqref{equ:five_term_exact_sequence} implies $\Ext^1_{ \cM }( C_j X,N ) = 0$.  Hence $C_j( X ) \in {}^{ \perp }\cC$.

Assume that $L_1C_j( X ) \neq 0$.  Pick an injective object $N$ which is in $\cC$ and satisfies $\Hom_{ \cM }( L_1C_j X,N ) \neq 0$.  By the previous paragraph, the term in \eqref{equ:five_term_exact_sequence} which involves $\Ext_{ \cX }^1$ is $0$.  However, the term involving $\Ext^2_{ \cM }$ is also $0$ since $N$ is injective, so \eqref{equ:five_term_exact_sequence} implies $\Hom_{ \cM }( L_1 C_j X,N ) = 0$.  This is a contradiction, so we conclude $L_1 C_j( X ) = 0$.  Combining with the previous paragraph shows $X \in \Phi( {}^{ \perp }\cC )$.  

Part (ii): Proved dually to part (i).
\end{proof}

\begin{Theorem}
\label{thm:Phi_Psi}
There are cotorsion pairs $\big( \Phi( \cA ),\Phi( \cA )^{ \perp } \big)$ and $\big( {}^{ \perp }\Psi( \cB ),\Psi( \cB ) \big)$ in $\cX$.
\end{Theorem}

\begin{proof}
The class $\cA$ contains the projective objects of $\cM$, and the class $\cB$ contains the injective objects of $\cM$.  Since we also have $\cA = {}^{ \perp }\cB$ and $\cB = \cA^{ \perp }$, Lemma \ref{lem:perps} implies
\[
  \Phi( \cA ) = {}^{ \perp }\{\, S_*( \cB ) \,\}
  \;\;,\;\;
  \Psi( \cB ) = \{\, S_*( \cA ) \,\}^{ \perp }.
\]
Hence there are the following cotorsion pairs, see \cite[def.\ 5.15]{GT}:
\begin{align*}
  \big( \Phi( \cA ),\Phi( \cA )^{ \perp } \big)
  & =
  \big( {}^{ \perp }\{\, S_*( \cB ) \,\},( {}^{ \perp }\{\, S_*( \cB ) \,\} )^{ \perp } \big), \\[1mm]
  \big( {}^{ \perp }\Psi( \cB ),\Psi( \cB ) \big)
  & =
  \big( {}^{ \perp }( \{\, S_*( \cA ) \,\}^{ \perp } ), \{\, S_*( \cA ) \,\}^{ \perp } \big). \qedhere
\end{align*}
\end{proof}

\begin{Theorem}
\label{thm:hereditary}
Assume that $( \cA,\cB )$ is a hereditary cotorsion pair in $\cM$.
\begin{enumerate}
\setlength\itemsep{4pt}

  \item  If $L_2C_j( \cE_L ) = 0$ for $j \in J$, then there is a hereditary cotorsion pair $\big( \Phi( \cA ),\Phi( \cA )^{ \perp } \big)$ in $\cX$.

  \item  If $R^2K_j( \cE_R ) = 0$ for $j \in J$, then there is a hereditary cotorsion pair $\big( {}^{ \perp }\Psi( \cB ),\Psi( \cB ) \big)$ in $\cX$.

\end{enumerate}
\end{Theorem}

\begin{proof}
The cotorsion pairs exist by Theorem \ref{thm:Phi_Psi}, and we must prove heredity under the given assumptions.

(i):  Lemma \ref{lem:hereditary} implies that $\cA$ is resolving, and that it is enough to prove that so is $\Phi( \cA )$.  Let $0 \rightarrow X' \rightarrow X \rightarrow X'' \rightarrow 0$ be a short exact sequence in $\cX$ with $X,X'' \in \Phi( \cA )$, and let $j \in J$ be given.  By definition we have $C_j( X ), C_j( X'' ) \in \cA$ and $L_1 C_j( X ) = L_1 C_j( X'' ) = 0$.  In particular, $X'' \in \cE_L$, so the assumption in part (i) says $L_2 C_j( X'' ) = 0$.  Hence the long exact sequence
\[
  \cdots
  \rightarrow L_2 C_j( X'' )
  \rightarrow L_1 C_j( X' )
  \rightarrow L_1 C_j( X )
  \rightarrow L_1 C_j( X'' )
  \rightarrow C_j( X' )
  \rightarrow C_j( X )
  \rightarrow C_j( X'' )
  \rightarrow 0
\]
reads
\[
  \cdots
  \rightarrow 0
  \rightarrow L_1 C_j( X' )
  \rightarrow 0
  \rightarrow 0
  \rightarrow C_j( X' )
  \rightarrow C_j( X )
  \rightarrow C_j( X'' )
  \rightarrow 0.
\]
This implies $L_1 C_j( X' ) = 0$.  It also implies $C_j( X' ) \in \cA$ because $\cA$ is resolving and $C_j( X ), C_j( X'' ) \in \cA$.  Hence $X' \in \Phi( \cA )$ as desired.

(ii):  Proved dually to (i).
\end{proof}

\begin{Definition}
\label{def:ExSeq}
Consider the following conditions on the classes $\cE_L$, $\cE_R$, $\Phi( \cA )$, $\Psi( \cB )$ from Definition \ref{def:Phi_Psi}.
\begin{VarDescription}{Seq\quad}
\setlength\itemsep{5pt}

  \item[(Ex)] $\cE_L = \cE_R$.

  \item[(Seq)] If $j \in J$ is given, then:

\begin{enumerate}
\setlength\itemsep{4pt}

  \item  Each $A \in \cA$ permits a short exact sequence in $\cX$,
\[  
  0 \rightarrow S_j( A ) \rightarrow R \rightarrow U \rightarrow 0,
\]
with $R \in \Phi( \cA )$ and $\Ext^2_{ \cX }( U,\cE_R ) = 0$.

  \item  Each $B \in \cB$ permits a short exact sequence in $\cX$,
\[
  0 \rightarrow W \rightarrow T \rightarrow S_j( B ) \rightarrow 0,
\]  
with $T \in \Psi( \cB )$ and $\Ext^2_{ \cX }( \cE_L,W ) = 0$.

\end{enumerate}

\end{VarDescription}
\end{Definition}

\begin{Remark}
It is not obvious that the sequences in condition (Seq) of the definition exist.  Their construction in the category of representations of a self-injective quiver is a key technical part of the paper, see Section \ref{sec:Seq}.
\end{Remark}

\begin{Theorem}
\label{thm:CTCP}
Assume that conditions {\rm (Comp1)}, {\rm (Ex)},  and {\rm (Seq)} hold (see Definitions \ref{def:cotorsion} and \ref{def:ExSeq}).  Then there are compatible cotorsion pairs $( \Phi( \cA ),\Phi( \cA )^{ \perp } )$ and $( {}^{ \perp }\Psi( \cB ),\Psi( \cB ) )$ in $\cX$.
\end{Theorem}

\begin{proof}
The cotorsion pairs exist by Theorem \ref{thm:Phi_Psi}, and we must prove that they are compatible under the given assumptions, which amounts to proving that condition (Comp2) holds.  We have assumed condition (Ex), so write $\cE = \cE_L = \cE_R$.  It is enough to prove 
\begin{align}
\label{equ:CTCPi}
  {}^{ \perp }\Psi( \cB ) \cap \cE & = \Phi( \cA ), \\
\label{equ:CTCPii}
  \Phi( \cA )^{ \perp } \cap \cE & = \Psi( \cB ),
\end{align}
since then
\[
  \Phi( \cA ) \cap \Phi( \cA )^{ \perp } \cap \cE
  = \Phi( \cA ) \cap \Psi( \cB )
  = {}^{ \perp }\Psi( \cB ) \cap \Psi( \cB ) \cap \cE,
\]
and this shows (Comp2) since $\cE$ can be removed from the displayed formula because $\Phi( \cA ),\Psi( \cB ) \subseteq \cE$.

We prove Equation \eqref{equ:CTCPi} by establishing the two inclusions.   

The inclusion $\subseteq$: Let $X \in {}^{ \perp }\Psi( \cB ) \cap \cE$ be given.  Given $j \in J$ and $B \in \cB$, condition (Seq)(ii) provides a short exact sequence in $\cX$,
\[  
  0 \rightarrow W \rightarrow T \rightarrow S_j( B ) \rightarrow 0,
\]
with $T \in \Psi( \cB )$ and $\Ext^2_{ \cX }( \cE,W ) = 0$.  There is a long exact sequence containing
\[
  \Ext^1_{ \cX }( X,T )
  \rightarrow \Ext^1_{ \cX }( X,S_j B )
  \rightarrow \Ext^2_{ \cX }( X,W ).
\]
The first term is zero since $X \in {}^{ \perp }\Psi( \cB )$ and $T \in \Psi( \cB )$.  The last term is zero since $X \in \cE$ and $\Ext^2_{ \cX }( \cE,W ) = 0$.  Hence the middle term is zero: $\Ext^1_{ \cX }( X,S_j B ) = 0$.  By Lemma \ref{lem:five_term_exact_sequence} this implies $\Ext^1_{ \cM }( C_j X,B ) = 0$ whence $C_j ( X ) \in {}^{ \perp }\cB = \cA$.  We also know $X \in \cE$, so $L_1C_j( X ) = 0$.  It follows that $X \in \Phi( \cA )$.

The inclusion $\supseteq$:  This follows because $\Phi( \cA ) \subseteq \cE$, while condition (Comp1) implies $\Phi( \cA ) \subseteq {}^{ \perp }\Psi( \cB )$.

Equation \eqref{equ:CTCPii} is proved dually to Equation \eqref{equ:CTCPi}.
\end{proof}

We end by recording a lemma which has almost the same proof as Theorem \ref{thm:Phi_Psi}.

\begin{Lemma}
\label{lem:Phi_Psi}
Let $\cC$ be a class of objects in $\cM$.
\begin{enumerate}
\setlength\itemsep{4pt}

  \item  Assume that for each non-zero $M \in \cM$ there is an injective object $I$ which is in $\cC$ and satisfies $\Hom_{ \cM }( M,I ) \neq 0$.

\medskip
\noindent  
  If $( \cA,\cB )$ is the cotorsion pair in $\cM$ cogenerated by $\cC$, then $\big( \Phi( \cA ),\Phi( \cA )^{ \perp } \big)$ is the cotorsion pair in $\cX$ cogenerated by $\{\, S_*( \cC ) \,\}$.

  \item  Assume that for each non-zero $M \in \cM$ there is a projective object $P$ which is in $\cC$ and satisfies $\Hom_{ \cM }( P,M ) \neq 0$.

\medskip
\noindent    
  If $( \cA,\cB )$ is the cotorsion pair in $\cM$ generated by $\cC$, then $\big( {}^{ \perp }\Psi( \cB ),\Psi( \cB ) \big)$ is the cotorsion pair in $\cX$ generated by $\{\, S_*( \cC ) \,\}$.

\end{enumerate}
\end{Lemma}

\begin{proof}
(i):  If $( \cA,\cB )$ is cogenerated by $\cC$, then $\cA = {}^{ \perp }\cC$, so Lemma \ref{lem:perps} implies $\Phi( \cA ) = {}^{ \perp }\{\, S_*( \cC ) \,\}$.  Hence
\[
  \big( \Phi( \cA ),\Phi( \cA )^{ \perp } \big)
  =
  \big( {}^{ \perp }\{\, S_*( \cC ) \,\},( {}^{ \perp }\{\, S_*( \cC ) \,\} )^{ \perp } \big) \\
\]
is the cotorsion pair cogenerated by $\{\, S_*( \cC ) \,\}$.

(ii) is proved dually to (i).
\end{proof}

\section{Functor categories}
\label{sec:functor_category}

This section introduces functor categories.  In particular, Paragraph \ref{bfhpg:quivers} explains how a category of quiver representations can be viewed as a functor category, whence Theorem \ref{thm:self-injective_main} has Theorem \ref{thm:main} as a special case.

\begin{Setup}
\label{set:blanket10}
Section \ref{sec:functor_category} uses the following setup.
\begin{itemize}
\setlength\itemsep{4pt}

  \item  $k$ is a field.
  
  \item  $R$ is a $k$-algebra.

\end{itemize}
\end{Setup}

\begin{bfhpg}[\bf Functor categories]
\label{bfhpg:functor_categories}
Let $Q$ be a small $k$-preadditive category; that is, each $\Hom$-space is a $k$-vector space and composition of morphisms is $k$-bilinear. 
The homomorphism functor and the radical of $Q$ will be denoted $Q( -,- )$ and $\rad_Q( -,- )$, see \cite[sec.\ A.3]{ASS}, \cite[sec.\ 3.2]{GabRoi}, and \cite[p.\ 303]{Kelly}.  

Let $Q^{\opp}$ denote the opposite category, and let ${}_{ k }\!\Mod$ and ${}_{ R }\!\Mod$ denote, respectively, the categories of $k$-vector spaces and $R$-left-modules.  There are the following functor categories.
\[
\mbox{
\begin{tabular}{cll}
  ${}_{ Q }\!\Mod$ & $=$ & the category of $k$-linear functors $Q \rightarrow {}_{ k }\!\Mod$ \\[1mm]
  $\Mod_Q$ & $=$ & the category of $k$-linear functors $Q^{ \opp } \rightarrow {}_{ k }\!\Mod$ \\[1mm]
  ${}_{ Q,R }\!\Mod$ & $=$ & the category of $k$-linear functors $Q \rightarrow {}_{ R }\!\Mod$ \\[1mm]
  ${}_{ Q }\!\Mod_Q$ & $=$ & the category of $k$-linear functors $Q^{ \opp } \times Q \rightarrow {}_{ k }\!\Mod$
\end{tabular}
}
\]
Their homomorphism functors are denoted $\Hom_Q$, $\Hom_{ Q^{ \opp } }$, $\Hom_{ Q,R }$, and $\Hom_{ Q^e }$.  

We think of them as the categories of $Q$-left-modules, $Q$-right-modules, $Q$-left-$R$-left-modules, and $Q$-bi-modules.  They are abelian categories with enough projective and injective objects, which are in fact Grothendieck categories.  In each of the categories ${}_{ Q }\!\Mod$, $\Mod_Q$, and ${}_{ Q,R }\!\Mod$, a sequence of functors $L' \xrightarrow{ \lambda' } L \xrightarrow{ \lambda } L''$ is short exact if $0 \xrightarrow{} L'( q ) \xrightarrow{ \lambda'_q } L( q ) \xrightarrow{ \lambda_q } L''( q ) \xrightarrow{} 0$ is a short exact sequence in ${}_{ k }\!\Mod$ or ${}_{ R }\!\Mod$ for each $q \in Q$.  An object $X$ of ${}_{ Q,R }\!\Mod$ can be viewed as an object of ${}_{ Q }\!\Mod$ by forgetting the $R$-structure on each $X( q )$ for $q \in Q$.  We refer to Appendix \ref{app:functor} for additional information.
\end{bfhpg}

\begin{Definition}
\label{def:FinRadSelfInj}
The following are conditions we can impose on a small $k$-preadditive category $Q$.
\begin{VarDescription}{SelfInj\quad}
\setlength\itemsep{5pt}

  \item[(Fin)]  Each $\Hom$-space of $Q$ is finite dimensional over $k$.  If $q \in Q$ is fixed then $Q( p,q ) = Q( q,p ) = 0$ except for finitely many $p \in Q$.  There is an integer $N$ such that $\rad_Q^N = 0$.
  
  \item[(Rad)]  If $q \in Q$ then $Q( q,q )$ is a local $k$-algebra, and the canonical map $k \xrightarrow{} Q( q,q )/\rad Q( q,q )$ is an isomorphism of $k$-algebras.  If $p \neq q$ are in $Q$ then $Q( p,q ) = \rad_Q( p,q )$.

  \item[(SelfInj)]  The category $Q$ has a Serre functor, that is, a $k$-linear autoequivalence $W: Q \rightarrow Q$ such that there are natural isomorphisms $Q( p,q ) \cong \dual\!Q( q,Wp )$ where $\dual( - ) = \Hom_k( -,k )$.

\end{VarDescription}
Note that the last part of condition (Rad) implies that different objects of $Q$ are non-isomorphic.  Conditions (Fin) and (Rad) imply that $Q$ is a locally bounded spectroid in the terminology of \cite[secs.\ 3.5 and 8.3]{GabRoi}, whence the functor categories over $Q$ share many properties of module categories over a finite dimensional algebra, see Appendix \ref{app:functor}.  If condition (SelfInj) also holds, then projective, injective, and flat objects coincide in each of ${}_{ Q }\!\Mod$ and $\Mod_Q$, see Paragraph \ref{bfhpg:flat_and_injective}.

\end{Definition}

\begin{bfhpg}
[\bf Special case: Quivers with relations]
\label{bfhpg:quivers}
Let $Q$ be a quiver with relations over $k$ in the sense of \cite[def.\ II.2.3]{ASS}.  Then $Q$ can be viewed as a small $k$-preadditive category:  The objects are the vertices, and the morphism spaces are the $k$-linear combinations of paths modulo relations.  Composition of morphisms is induced by concatenation of paths.

Viewed as a quiver with relations, $Q$ has a category $\cX$ of representations with values in ${}_{ R }\!\Mod$.  Viewed as a small $k$-preadditive category, $Q$ has the functor category ${}_{ Q,R }\!\Mod$.  The categories $\cX$ and ${}_{ Q,R }\!\Mod$ can be identified.

We say that $Q$ is a {\em self-injective quiver with relations} if $Q$, viewed as a small $k$-preadditive category, satisfies conditions (Fin), (Rad), and (SelfInj) of Definition \ref{def:FinRadSelfInj}.

The quivers with relations from the introduction are self-injective.  In particular, the Serre functors are given as follows: For \eqref{equ:Gillespie_quiver} by the shift $q \mapsto q-1$, for \eqref{equ:intro_quiver_2} by the shift $q \mapsto q-1$ where $q$ is taken modulo $N$, and for \eqref{equ:intro_quiver_4} by reflecting in a horizontal line through the vertices $( 1,j )$, then shifting one vertex to the right.
\end{bfhpg}

\begin{figure}
\begin{center}
\begin{tabular}{cc|cc}
  Structures over the category $Q$ &&&   Structures over the algebra $\Lambda$ \\[2mm] \hline
  ${}_{ Q,R }\!\Mod$ &&&   ${}_{ \Lambda,R }\!\Mod = \Lambda$-left-$R$-left-modules \\[1mm]
  ${}_{ Q }\!\Mod$ &&& ${}_{ \Lambda }\!\Mod = \Lambda$-left-modules \\[1mm]
  $\Mod_Q$ &&&   $\Mod_{ \Lambda } = \Lambda$-right-modules \\[1mm]
  ${}_{ Q }\!\Mod_Q$ &&& ${}_{ \Lambda }\!\Mod_{ \Lambda } = \Lambda$-bimodules \\[1mm]
  $\Hom_{ Q,R }$ &&& $\Hom_{ \Lambda,R } =$ homomorphisms of $\Lambda$-left-$R$-left-modules \\[1mm]
  $\Hom_Q$ &&& $\Hom_{ \Lambda } =$ homomorphisms of $\Lambda$-left-modules \\[1mm]
  $\Hom_{ Q^{ \opp } }$ &&& $\Hom_{ \Lambda^{ \opp } } =$ homomorphisms of $\Lambda$-right-modules \\[1mm]  
  $\Hom_{ Q^e }$ &&& $\Hom_{ \Lambda^e } =$ homomorphisms of $\Lambda$-bimodules \\[1mm]
  $\underset{ Q }{ \otimes }$ &&& $\underset{ \Lambda }{ \otimes } =$ tensor product of $\Lambda$-modules \\[3mm]
  $\underset{ k }{ \otimes }$ &&& $\underset{ k }{ \otimes } =$ tensor product of $k$-vector spaces \\[3mm]
  $S \langle q \rangle$ &&& $S \langle q \rangle =$ the simple $\Lambda$-left-module supported at vertex $q$ \\[1mm]
  $\dual\!S \langle q \rangle$ &&& $\dual\!S \langle q \rangle =$ the simple $\Lambda$-right-module supported at vertex $q$
\end{tabular}
\end{center}
\caption{A finite self-injective quiver with relations $Q$ can be viewed as a small $k$-preadditive category.  On the other hand, it gives a finite dimensional algebra $\Lambda$, and structures over $Q$ and $\Lambda$ can be identified as shown.}
\label{fig:dictionary}
\end{figure}

\begin{bfhpg}
[\bf Special case: Finite quivers with relations]
\label{bfhpg:finite_quivers}
Let $Q$ be a self-injective quiver with relations over $k$.  Assume that $Q$ is finite and connected, and that its relations are given by an admissible ideal $\fa$ in the path algebra $A$ over $k$, see \cite[def.\ II.2.1]{ASS}.

On the one hand, $Q$ can be viewed as a small $k$-preadditive category, which has the functor category ${}_{ Q,R }\!\Mod$.  On the other hand, there is a finite dimensional algebra $\Lambda = A/\fa$, which has the category ${}_{ \Lambda,R }\!\Mod$ of $\Lambda$-left-$R$-left-modules.  The categories ${}_{ Q,R }\!\Mod$ and ${}_{ \Lambda,R }\!\Mod$ can be identified.

A more extensive list of identifications is given in Figure \ref{fig:dictionary}, where the entries in the first column are explained in Paragraph \ref{bfhpg:functor_categories} and Appendix \ref{app:functor}.  The list can be extended with $\Ext$- and $\Tor$-functors.

Note that since $Q$ is a self-injective quiver with relations, $\Lambda$ is a self-injective algebra.
\end{bfhpg}

\section{Proof of Theorem \ref{thm:main}}
\label{sec:self-injective}

This section proves Theorem \ref{thm:self-injective_main}, which has Theorem \ref{thm:main} as a special case, see Paragraph \ref{bfhpg:quivers}.

Sections \ref{sec:self-injective} through \ref{sec:Seq} are phrased in the language of functor categories over a small $k$-preadditive category $Q$.  A reader who prefers modules instead of functors can use Figure \ref{fig:dictionary} to specialise everything to the case of modules over a finite dimensional self-injective algebra $\Lambda$.

\begin{Setup}
\label{set:self-injective}
Sections \ref{sec:self-injective} through \ref{sec:Seq} use the following setup, which dovetails with Setups \ref{set:blanket} and \ref{set:blanket10} so the results of Sections \ref{sec:Phi_Psi} and \ref{sec:functor_category} can be used verbatim.  We refer to Appendix \ref{app:functor} for additional information, in particular on several functors which will be used extensively: $\underset{ k }{ \otimes }$, $\Hom_k$, $\underset{ Q }{ \otimes }$, $\Tor^Q_i$, $\Hom_Q$, $\Ext_Q^i$, $\Hom_{ Q,R }$, $\Ext_{ Q,R }^i$.  
\begin{itemize}
\setlength\itemsep{4pt}

  \item  $k$ is a field.
  
  \item  $R$ is a $k$-algebra.

  \item  $Q$ is a small $k$-preadditive category satisfying conditions (Fin), (Rad), and (SelfInj) of Definition \ref{def:FinRadSelfInj}. 

  \item  $\cM = {}_{ R }\!\Mod$ is the category of $R$-left-modules. 

  \item  $\cX = {}_{ Q,R }\!\Mod$ is the category of $k$-linear functors $Q \rightarrow {}_{ R }\!\Mod$.  
  
\smallskip
\noindent
The categories $\cM$ and $\cX$ have enough projective and enough injective objects by Paragraph \ref{bfhpg:simple}.  

  \item  $ ( \cA,\cB )$ is a cotorsion pair in $\cM$.

  \item  $J = \obj\,Q$.  The statement $q \in \obj\,Q$ will be abbreviated $q \in Q$.  

  \item  
For each $q \in Q$, there is a simple object $\dual\!S \langle q \rangle \in \Mod_Q$ and a simple object $S \langle q \rangle \in {}_{ Q }\!\Mod$, see Paragraph \ref{bfhpg:simple}(i).  The functors 
\[
\vcenter{
\xymatrix
{
  *+[l]{ {}_{ R }\!\Mod = \cM }
    \ar[rrr]^{ S_q } &&&
  *+[r]{ \cX = {}_{ Q,R }\!\Mod }
    \ar@/_2.0pc/[lll]_{ C_q }
    \ar@/^2.0pc/[lll]^{ K_q }
}
        }
\]
are defined by:
\begin{align}
\nonumber
  C_q( - ) & = \dual\!S \langle q \rangle \underset{ Q }{ \otimes } -, \\[1mm]
\label{equ:self-injective_CSK}
  S_q( - ) & = S \langle q \rangle \underset{ k }{ \otimes } -, \\[1mm]
\nonumber
  K_q( - ) & = \Hom_Q( S \langle q \rangle,- ).
\end{align}
There is an adjoint pair $( C_q,S_q )$ by Paragraph \ref{bfhpg:standard}(ii) and the observation that we have $S_q( - ) = \Hom_k( \dual\!S \langle q \rangle,- )$ by Paragraph \ref{bfhpg:standard}(vi).  There is an adjoint pair $( S_q,K_q )$ by Paragraph \ref{bfhpg:standard}(i).

  \item  $\cE$ denotes either $\cE_L$ or $\cE_R$; these classes are equal because condition (Ex) of Definition \ref{def:ExSeq} holds by Proposition \ref{pro:self-injective_Ex}.

\end{itemize}
\end{Setup}

\begin{Theorem}
\label{thm:self-injective_main}
If $( \cA,\cB )$ is a hereditary cotorsion pair in $\cM = {}_{ R }\!\Mod$, then there are hereditary, compatible cotorsion pairs $\big( \Phi( \cA ),\Phi( \cA )^{ \perp } \big)$ and $\big( {}^{ \perp }\Psi( \cB ),\Psi( \cB ) \big)$ in $\cX = {}_{ Q,R }\!\Mod$, where
\begin{align*}
  \Phi( \cA ) & = \{\, X \in \cX \mid \mbox{If $q \in Q$ then $C_q( X ) \in \cA$ and $L_1C_q( X ) = 0$} \,\}, \\[1mm]
  \Psi( \cB ) & = \{\, X \in \cX \mid \mbox{If $q \in Q$ then $K_q( X ) \in \cB$ and $R^1K_q( X ) = 0$} \,\}.
\end{align*}
\end{Theorem}

\begin{proof}
The formulae for $\Phi( \cA )$ and $\Psi( \cB )$ are those of Definition \ref{def:Phi_Psi} adapted to the present setup, so the results of Section \ref{sec:Phi_Psi} apply.  In particular, the cotorsion pairs exist by Theorem \ref{thm:Phi_Psi}.  They are hereditary by Theorem \ref{thm:hereditary} combined with Proposition \ref{pro:self-injective_Ex} below.  They are compatible by Theorem \ref{thm:CTCP} combined with Propositions \ref{pro:self-injective_Ex}, \ref{pro:self_injective_Comp1}, and \ref{pro:25} below.
\end{proof}

\begin{Theorem}
\label{thm:self-injective_complete}
We have:
\begin{enumerate}
\setlength\itemsep{4pt}

  \item  If the cotorsion pair $( \cA,\cB )$ is generated by a set, then the cotorsion pair $\big( {}^{ \perp }\Psi( \cB ),\Psi( \cB ) \big)$ is complete.

  \item  Assume that $Q$ arises from the special case described in Paragraph \ref{bfhpg:finite_quivers}, so corresponds to a finite dimensional self-injective $k$-algebra $\Lambda$.  If $\cA$ is closed under pure quotients, then the cotorsion pair $\big( \Phi( \cA ),\Phi( \cA )^{ \perp } \big)$ is complete.

\end{enumerate}
\end{Theorem}

\begin{proof}
(i)  Suppose that $\cC$ is a set of objects of $\cM$ which generates $( \cA,\cB )$.  This is still the case after adding the projective $R$-left-module ${}_{ R }R$ to $\cC$.  Then Lemma \ref{lem:Phi_Psi}(ii) says that $\{ S_*( \cC ) \}$ is a set of objects of $\cX$ which generates $\big( {}^{ \perp }\Psi( \cB ),\Psi( \cB ) \big)$.  This cotorsion pair is hence complete by \cite[lem.\ 5.20]{GT} and \cite[def.\ 3.11, prop.\ 3.13, and prop.\ 5.8]{Sto}.  Note that the proof of \cite[lem.\ 5.20]{GT} goes through for $\cX$ because it has enough projective objects and enough injective objects by Paragraph \ref{bfhpg:simple}.

(ii)  As explained in Paragraph \ref{bfhpg:finite_quivers}, we can view $\big( \Phi( \cA ),\Phi( \cA )^{ \perp } \big)$ as a cotorsion pair in ${}_{ \Lambda,R }\!\Mod$, hence as a cotorsion pair in ${}_{ \Lambda \otimes R }\!\Mod$, the category of left-modules over the $k$-algebra $\Lambda \underset{ k }{ \otimes } R$, which will be denoted simply by $\Lambda \otimes R$.  By \cite[lem.\ 5.13(b) and lem.\ 5.20]{GT} and \cite[thm.\ 2.5]{HJ2} it is enough to show that $\Phi( \cA )$ is closed under pure quotients.

Let
\[
  \sigma \;\;=\;\;
  0 \xrightarrow{} X' \xrightarrow{} X \xrightarrow{} X'' \xrightarrow{} 0
\]  
be pure short exact in ${}_{ \Lambda \otimes R }\!\Mod$ with $X \in \Phi( \cA )$.  We will show $X'' \in \Phi( \cA )$, that is, $C_q( X'' ) \in \cA$ and $L_1C_q( X'' ) = 0$ for each $q \in Q$.

Given $q \in Q$ and $B \in \Mod_R$ we have that
\[
  B \underset{ R }{ \otimes } C_q( \sigma ) =
  B \underset{ R }{ \otimes } ( \dual\!S \langle q \rangle \underset{ \Lambda }{ \otimes } \sigma )
  \cong
  ( B \underset{ k }{ \otimes } \dual\!S \langle q \rangle ) \underset{ \Lambda \otimes R }{ \otimes } \sigma
\]
is exact, because $\sigma$ is pure short exact in ${}_{ \Lambda \otimes R }\!\Mod$.  The isomorphism is by \cite[prop.\ IX.2.1]{CE}.  Hence
\[
  C_q( \sigma ) \;\;=\;\;
  0 \xrightarrow{} C_q( X' ) \xrightarrow{} C_q( X ) \xrightarrow{} C_q( X'' ) \xrightarrow{} 0
\]  
is pure short exact in ${}_{ R }\!\Mod$.  But $X \in \Phi( \cA )$ implies $C_q( X ) \in \cA$ whence $C_q( X'' ) \in \cA$ since $\cA$ is closed under pure quotients.

Given $M \in \Mod_{ \Lambda }$ we have that
\[
  M \underset{ \Lambda }{ \otimes } \sigma
  \cong
  M \underset{ \Lambda }{ \otimes } \big( ( \Lambda \otimes R ) \underset{ \Lambda \otimes R }{ \otimes } \sigma \big)
  \cong
  \big( M \underset{ \Lambda }{ \otimes }( \Lambda \otimes R ) \big) \underset{ \Lambda \otimes R }{ \otimes } \sigma
\]
is exact, because $\sigma$ is pure short exact in ${}_{ \Lambda \otimes R }\!\Mod$.  Hence $\sigma$, viewed as a sequence of $\Lambda$-left-modules, is pure short exact.  But $X \in \Phi( \cA )$ implies $X \in \cE$, so $X$, viewed as a $\Lambda$-left-module, is flat by Proposition \ref{pro:self-injective_Ex}.  It follows that $X''$, viewed as a $\Lambda$-left-module, is flat, so $L_1C_q( X'' ) = 0$.
\end{proof}

We end this section with a description of the trivial objects in the model category structure on $\cX$ provided by Theorems \ref{thm:Hovey} and \ref{thm:self-injective_main}.  We thank an anonymous referee for drawing attention to this question.

\begin{Theorem}
\label{thm:W}
If $( \cA,\cB )$ is a hereditary cotorsion pair in $\cM = {}_{ R }\!\Mod$, and either of the cotorsion pairs $\big( \Phi( \cA ),\Phi( \cA )^{ \perp } \big)$ and $\big( {}^{ \perp }\Psi( \cB ),\Psi( \cB ) \big)$ of Theorem \ref{thm:self-injective_main} is complete (see Theorem \ref{thm:self-injective_complete}), then the class $\cW$ of trivial objects (defined in Theorem \ref{thm:Hovey}) satisfies $\cW = \cE$.
\end{Theorem}

\begin{proof}
The inclusion $\cW \subseteq \cE$:  By the first part of the first formula of Theorem \ref{thm:Hovey}, each $W \in \cW$ sits in a short exact sequence $0 \xrightarrow{} P \xrightarrow{} F \xrightarrow{} W \xrightarrow{} 0$ in $\cX$ with $P \in \Psi( \cB )$, $F \in \Phi( \cA )$.  We have $P,F \in \cE$ by Definition \ref{def:Phi_Psi}.  By Proposition \ref{pro:self-injective_Ex} this means that $P$ and $F$ are injective when viewed as objects of ${}_{Q}\Mod$.  The same is hence true for $W$, so $W \in \cE$ by Proposition \ref{pro:self-injective_Ex} again.

The inclusion $\cW \supseteq \cE$:  First, suppose $\big( \Phi( \cA ),\Phi( \cA )^{ \perp } \big)$ is complete.  Given $E \in \cE$ there is a short exact sequence $0 \xrightarrow{} Q \xrightarrow{} F \xrightarrow{} E \xrightarrow{} 0$ in $\cX$ with $Q \in \Phi( \cA )^{ \perp }$, $F \in \Phi( \cA )$.  By the first part of the first formula of Theorem \ref{thm:Hovey} it is enough to show $Q \in \Psi( \cB )$.  We have $F \in \cE$ by Definition \ref{def:Phi_Psi}.  Proposition \ref{pro:self-injective_Ex} says that $F$ and $E$ are projective when viewed as objects of ${}_{Q}\Mod$.  The same is hence true for $Q$, so $Q \in \cE$ by Proposition \ref{pro:self-injective_Ex} again.  Hence $Q \in \Phi( \cA )^{ \perp } \cap \cE$.  It follows that, as desired, $Q \in \Psi( \cB )$, see Equation \eqref{equ:CTCPii} in the proof of Theorem \ref{thm:CTCP}.  The theorem applies because conditions (Comp1), (Ex), and (Seq) (see Definitions \ref{def:cotorsion}(iv) and \ref{def:ExSeq}) hold by Propositions \ref{pro:self-injective_Ex}, \ref{pro:self_injective_Comp1}, and \ref{pro:25} below.

Secondly, suppose $\big( {}^{ \perp }\Psi( \cB ),\Psi( \cB ) \big)$ is complete.  Then $\cW \supseteq \cE$ is proved similarly using the second part of the first formula of Theorem \ref{thm:Hovey} and Equation \eqref{equ:CTCPi}. 
\end{proof}

\section{Condition (Ex)}
\label{sec:Ex}

Section \ref{sec:Ex} continues to use Setup \ref{set:self-injective}.  The aim is to prove Proposition \ref{pro:self-injective_Ex}, by which condition (Ex) of Definition \ref{def:ExSeq}
holds.  This is used in the proof of Theorem \ref{thm:self-injective_main}.  We also establish some other properties of the class $\cE$.

\begin{Lemma}
\label{lem:self-injective_standard_isomorphisms}
If $i \geqslant 0$ and $q \in Q$ then 
\begin{enumerate}
\setlength\itemsep{4pt}

  \item  $L_iC_q( - ) = \Tor^Q_i( \dual\!S \langle q \rangle,- )$,

  \item  $R^iK_q( - ) = \Ext_Q^i( S \langle q \rangle,- )$,

\end{enumerate}
and there are isomorphisms in ${}_{ R }\!\Mod$,
\begin{enumerate}
\setcounter{enumi}{2}
\setlength\itemsep{4pt}

  \item  $L_iC_q( M \underset{ k }{ \otimes } B ) \cong \Tor^Q_i( \dual\! S \langle q \rangle,M ) \underset{ k }{ \otimes } B$,

  \item  $R^i K_q( M \underset{ k }{ \otimes } B ) \cong \Ext_Q^i( S \langle q \rangle, M ) \underset{ k }{ \otimes } B$,

\end{enumerate}
natural in $M \in {}_{ Q }\!\Mod$ and $B \in {}_{ R }\!\Mod$.
\end{Lemma}

\begin{proof}
Parts (i) and (ii) follow from Equation \eqref{equ:self-injective_CSK} and Paragraph \ref{bfhpg:Tor_and_Ext}.  Parts (iii) and (iv) follow from parts (i) and (ii) combined with Paragraph \ref{bfhpg:Tor_and_Ext}, parts (iii) and (ii).
\end{proof}

\begin{Proposition}
\label{pro:self-injective_Ex}
In the situation of Setup \ref{set:self-injective}, condition (Ex) of Definition \ref{def:ExSeq} holds, that is $\cE_L = \cE_R$.  We write $\cE = \cE_L = \cE_R$ and have
\begin{align*}
  \cE
  & = \{\, X \in \cX \mid \mbox{$X$ is projective when viewed as an object of ${}_{ Q }\!\Mod$} \,\} \\[1mm]
  & = \{\, X \in \cX \mid \mbox{$X$ is flat when viewed as an object of ${}_{ Q }\!\Mod$} \,\} \\[1mm]
  & = \{\, X \in \cX \mid \mbox{$X$ is injective when viewed as an object of ${}_{ Q }\!\Mod$} \,\}.
\end{align*}
\end{Proposition}

\begin{proof}
Combining Definition \ref{def:Phi_Psi} and Lemma \ref{lem:self-injective_standard_isomorphisms}(i) shows
\[
  \cE_L
  = \{\, X \in \cX \mid \mbox{If $q \in Q$ then $\Tor^Q_1( \dual\!S \langle q \rangle,X ) = 0$} \,\}.
\]
Combining this with Equation \eqref{equ:criterion_for_flatness} proves
\[
  \cE_L = \{\, X \in \cX \mid \mbox{$X$ is flat when viewed as an object of ${}_{ Q }\!\Mod$} \,\}.
\]
Similarly, combining Definition \ref{def:Phi_Psi}, Lemma \ref{lem:self-injective_standard_isomorphisms}(ii), and Equation \eqref{equ:criterion_for_injectivity} proves 
\[
  \cE_R = \{\, X \in \cX \mid \mbox{$X$ is injective when viewed as an object of ${}_{ Q }\!\Mod$} \,\}.
\]
The proposition now follows from Equation \eqref{equ:flat_v_injective}.
\end{proof}

\begin{Lemma}
\label{lem:24}
If $M \in {}_{ Q }\!\Mod$ has finite length and $I \in {}_{ R }\!\Mod$ is injective, then $M \underset{ k }{ \otimes } I \in \cE^{ \perp }$.  
\end{Lemma}

\begin{proof}
Let $E \in \cE$ have the projective resolution $P$ in $\cX = {}_{ Q,R }\!\Mod$.  Then
\begin{align*}
  \Ext_{ Q,R }^1 \big( E,\Hom_k( \dual\!S \langle q \rangle,I ) \big)
  & \cong
  \H^1 \Hom_{ Q,R }\big( P,\Hom_k( \dual\!S \langle q \rangle,I ) \big) \\
  & \stackrel{ \rm (a) }{ \cong }
  \H^1 \Hom_R ( \dual\!S \langle q \rangle \underset{ Q }{ \otimes } P,I ) \\
  & \stackrel{ \rm (b) }{ \cong }
  \Hom_R \big( \H_1( \dual\!S \langle q \rangle \underset{ Q }{ \otimes } P),I \big) \\
  & \stackrel{ \rm (c) }{ \cong } \Hom_R \big( \Tor^Q_1( \dual\!S \langle q \rangle,E ),I \big) \\
  & \stackrel{ \rm (d) }{ \cong }
  \Hom_R \big( 0,I \big) \\
  & = 0.
\end{align*}
Here (a) is by Paragraph \ref{bfhpg:standard}(ii) and (b) is because $I$ is injective.  The isomorphism (c) is because $P$ consists of projective objects in $\cX$, which are also projective when viewed as objects of ${}_{ Q }\!\Mod$ by Paragraph \ref{bfhpg:simple}(iv).  Finally, (d) is by Proposition \ref{pro:self-injective_Ex}.

Hence $\Hom_k( \dual\!S \langle q \rangle,I ) \in \cE^{ \perp }$, and Paragraph \ref{bfhpg:standard}(vi) gives
$S \langle q \rangle \underset{ k }{ \otimes } I \in \cE^{ \perp }$.  However, by Paragraph \ref{bfhpg:simple}(i) the object $M$ has a finite filtration with quotients of the form $S \langle q \rangle$ for $q \in Q$, so by Paragraph \ref{bfhpg:tensor_and_hom}(iv) the object $M \underset{ k }{ \otimes } I$ has a finite filtration with quotients of the form $S \langle q \rangle \underset{ k }{ \otimes } I$ for $q \in Q$, and it follows that $M \underset{ k }{ \otimes } I \in \cE^{ \perp }$ as claimed.  
\end{proof}

\begin{Lemma}
\label{lem:self-injective_E_perp}
We have $\cE^{ \perp } = \cE^{ \perp_{ \infty } }$ as subcategories of $\cX$ where $\cE^{ \perp_{ \infty } } = \{\, X \in \cX \mid \Ext_{ \cX }^{ \geqslant 1 }( \cE,X ) = 0 \,\}$.
\end{Lemma}

\begin{proof}
The proof of \cite[cor.\ 5.25]{GT} goes through for $\cX = {}_{ Q,R }\!\Mod$, so it is enough to see that $\cE$ is closed under syzygies.  Let $0 \rightarrow \Omega E \rightarrow P \rightarrow E \rightarrow 0$ be a short exact sequence in $\cX$ with $P$ projective and $E \in \cE$.  Then $P$ and $E$ are projective when viewed in ${}_{ Q }\!\Mod$, see Proposition \ref{pro:self-injective_Ex} and Paragraph \ref{bfhpg:simple}(iv).  Hence $\Omega E$ is projective when viewed in ${}_{ Q }\!\Mod$, so $\Omega E \in \cE$ by Proposition \ref{pro:self-injective_Ex}.
\end{proof}

\section{Condition (Comp1)}
\label{sec:Comp1}

Section \ref{sec:Comp1} continues to use Setup \ref{set:self-injective}.  The aim is to prove Proposition \ref{pro:self_injective_Comp1}, by which condition (Comp1) of Definition \ref{def:cotorsion}(iv) holds.  This is used in the proof of Theorem \ref{thm:self-injective_main}.

\begin{Lemma}
\label{lem:self-injective_SqA_in_perp_Psi}
If $q \in Q$ and $A \in \cA$, then $S_q( A ) \in {}^{ \perp }\Psi( \cB )$.
\end{Lemma}

\begin{proof}
The categories $\cM$ and $\cX$ have enough projective and injective objects by Paragraph \ref{bfhpg:simple}, and the functor $S_q( - ) = S \langle q \rangle \underset{ k }{ \otimes } -$ is exact by Paragraph \ref{bfhpg:tensor_and_hom}(iv).  Given $Y \in \Psi( \cB )$ we have $R^1K_q( Y ) = 0$ by definition of $\Psi( \cB )$, so Lemma \ref{lem:five_term_exact_sequence_dual} gives an isomorphism $\Ext_{ \cM }^1( A,K_qY ) \cong \Ext_{ \cX }^1( S_qA,Y )$.  The first $\Ext$ is zero since $K_q( Y ) \in \cB$ by definition of $\Psi( \cB )$, so the lemma follows. 
\end{proof}

\begin{Proposition}
\label{pro:self_injective_Comp1}
In the situation of Setup \ref{set:self-injective}, condition (Comp1) of Definition \ref{def:cotorsion} holds.
\end{Proposition}

\begin{proof}
If $X \in \Phi( \cA )$ is given, then $X \in \cE_L = \cE$ by definition of $\Phi( \cA )$, so Proposition \ref{pro:self-injective_Ex} says that $X$ is flat when viewed as an object of ${}_{ Q }\!\Mod$.  This means that $- \underset{ Q }{ \otimes } X$ is exact, so the filtration 
\eqref{equ:self-injective_radical_filtration} induces a filtration in $\cX = {}_{ Q,R }\!\Mod$:
\[
  0 = \rad_Q^N \underset{ Q }{ \otimes } X  \subseteq \cdots \subseteq \rad_Q^1 \underset{ Q }{ \otimes } X \subseteq \rad_Q^0 \underset{ Q }{ \otimes } X = X.
\]
The final equality is by Equation \eqref{equ:self-injective_identity_bimodule}. 
The quotients are
\begin{align*}
  ( \rad_Q^i/\rad_Q^{ i-1 } ) \underset{ Q }{ \otimes } X
  & \stackrel{ \rm (a) }{ \cong }\Bigg( \coprod_{ p,q \in Q } ( \dual\!S \langle p \rangle \underset{ k }{ \otimes } S \langle q \rangle )^{ n_i( p,q ) } \Bigg) \underset{ Q }{ \otimes } X \\
  & \stackrel{ \rm (b) }{ \cong } \coprod_{ p,q \in Q } \big( S \langle q \rangle \underset{ k }{ \otimes } (\dual\!S \langle p \rangle \underset{ Q }{ \otimes } X) \big)^{ n_i( p,q ) } \\
  & = \coprod_{ p,q \in Q } S_qC_p( X )^{ n_i( p,q ) } \\
  & = (*),
\end{align*}
where (a) is by Equation \eqref{equ:self-injective_radical_quotients}, while (b) uses that $- \underset{ Q }{ \otimes } X$ preserves coproducts, followed by Paragraph \ref{bfhpg:standard}(iv).  However, $C_p( X ) \in \cA$ by definition of $\Phi( \cA )$, so Lemma \ref{lem:self-injective_SqA_in_perp_Psi} implies $(*) \in {}^{ \perp }\Psi( \cB )$, whence also $X \in {}^{ \perp }\Psi( \cB )$ as desired.
\end{proof}

\section{Condition (Seq)}
\label{sec:Seq}

Section \ref{sec:Seq} continues to use Setup \ref{set:self-injective}.  The aim is to prove Proposition \ref{pro:25} by which condition (Seq) of Definition \ref{def:ExSeq} holds.  This is used in the proof of Theorem \ref{thm:self-injective_main}.

\begin{Setup}
\label{set:self-injective2}
In addition to Setup \ref{set:self-injective}, Section \ref{sec:Seq} uses the following setup.
\begin{itemize}
\setlength\itemsep{4pt}

  \item  $M_0 \in {}_{ Q }\!\Mod$ is an object of finite length.  By Paragraph \ref{bfhpg:simple}(ii) it has an augmented minimal projective resolution, which we break into short exact sequences as follows.
\[
\vcenter{
  \xymatrix {
    \cdots \ar [rr] && P_2 \ar^{ \partial^P_2 }[rr] \ar@{->>}_{ \pi_2 }[rd] && P_1 \ar^{ \partial^P_1 }[rr] \ar@{->>}_{ \pi_1 }[rd] && P_0 \ar@{->>}^{ \pi_0 }[rr] && M_0 \\
    &&& M_2 \ar@{^(->}_{ \mu_2}[ur] && M_1 \ar@{^(->}_{ \mu_1}[ur] \\
            }
        }
\]
Each $P_i$ and each $M_i$ has finite length, and for each $q \in Q$, the functors $\dual\!S \langle q \rangle \underset{ Q }{ \otimes } -$ and $\Hom_Q( S \langle q \rangle,- )$ vanish on the $\partial^P_i$.

  \item  $B^0 \in {}_{ R }\!\Mod$ is a module with an augmented injective resolution, which we break into short exact sequences as follows.
\[
\vcenter{
  \xymatrix {
    B^0 \ar@{^(->}^{ \beta^0 }[rr] && I^0 \ar^{ \partial_I^0 }[rr] \ar@{->>}_{ \alpha^0 }[dr] && I^1 \ar^{ \partial_I^1 }[rr] \ar@{->>}_{ \alpha^1 }[dr] && I^2 \ar[rr] && \cdots \\
    &&& B^1 \ar@{^(->}_{ \beta^1 }[ur] && B^2 \ar@{^(->}_{ \beta^2 }[ur] \\
            }
        }
\]

\end{itemize}
\end{Setup}

\begin{Construction}
\label{con:self-injective}
This construction consists of two parts labelled (i) and (ii).

(i)  For each $i \geqslant 0$ we define a short exact sequence
\begin{equation}
\label{equ:self-injective_fundamental_sequence}
  0 \xrightarrow{}
  E^i \xrightarrow{ \varepsilon^i }
  T^i \xrightarrow{ \tau^i }
  M_i \underset{ k }{ \otimes } B^i \xrightarrow{}
  0
\end{equation}
in $\cX$ as follows:

If $i = 0$ then \eqref{equ:self-injective_fundamental_sequence} is defined to be
\[
  0 \xrightarrow{}
  0 \xrightarrow{}
  M_0 \underset{ k }{ \otimes } B^0 \xrightarrow{ \id }
  M_0 \underset{ k }{ \otimes } B^0 \xrightarrow{}
  0.
\] 
If \eqref{equ:self-injective_fundamental_sequence} has already been defined for a given value $i \geqslant 0$, then we use it as the last non-trivial column of the following diagram.  The lower right square is a pullback, and the rows and columns are exact, see Paragraph \ref{bfhpg:tensor_and_hom}(iv).
\begin{equation}
\label{equ:pullback}
\vcenter{
  \xymatrix @-0.9pc {
    &&&& 0 \ar[dd] && 0 \ar[dd] \\
    \\  
    &&&& E^i \ar@{=}[rr] \ar_{ \zeta^i }[dd] && E^i \ar^{ \varepsilon^i }[dd] \\
    \\
    0 \ar[rr] && M_{ i+1 } \underset{ k }{ \otimes } B^i \ar^-{ \theta^{ i+1 }}[rr] \ar@{=}[dd] && E^{ i+1 } \ar^{ \kappa^{ i+1 }}[rr] \ar_{ \eta^i }[dd] && T^i \ar[rr] \ar^{ \tau^i }[dd] && 0 \\
    \\
    0 \ar[rr] && M_{ i+1 } \underset{ k }{ \otimes } B^i \ar_-{ \mu_{ i+1 } \otimes B^i }[rr] && P_i \underset{ k }{ \otimes } B^i \ar_-{ \pi_i \otimes B^i }[rr] \ar[dd] && M_i \underset{ k }{ \otimes } B^i \ar[rr] \ar[dd] && 0 \\
    \\
    &&&& 0 && 0 \\
            }
        }
\end{equation}
The row which contains $E^{ i+1 }$ is used as the first non-trivial row of the following diagram.  The upper left square is a pushout, and the rows and columns are exact.
\begin{equation}
\label{equ:pushout}
\vcenter{
  \xymatrix @-0.98pc {
    && 0 \ar[dd] && 0 \ar[dd] \\
    \\
    0 \ar[rr] && M_{ i+1 } \underset{ k }{ \otimes } B^i \ar^-{ \theta^{ i+1 }}[rr] \ar_{ M_{ i+1 } \otimes \beta^i }[dd] && E^{ i+1 } \ar^{ \kappa^{ i+1 }}[rr] \ar^{ \varepsilon^{ i+1 }}[dd] && T^i \ar[rr] \ar@{=}[dd] && 0 \\
    \\
    0 \ar[rr] && M_{ i+1 } \underset{ k }{ \otimes } I^i \ar_-{ \gamma^{ i+1 }}[rr] \ar_{ M_{ i+1 } \otimes \alpha^i }[dd] && T^{ i+1 } \ar_{ \delta^{ i+1 }}[rr] \ar^{ \tau^{ i+1 }}[dd] && T^i \ar[rr] && 0 \\
    \\
    && M_{ i+1 } \underset{ k }{ \otimes } B^{ i+1 } \ar@{=}[rr] \ar[dd] && M_{ i+1 } \underset{ k }{ \otimes } B^{ i+1 }  \ar[dd] \\
    \\
    && 0 && 0 \\
            }
        }
\end{equation}
The column which contains $E^{ i+1 }$ defines \eqref{equ:self-injective_fundamental_sequence} for $i+1$.  Note that Diagrams \eqref{equ:pullback} and \eqref{equ:pushout} define a number of morphisms in addition to those in \eqref{equ:self-injective_fundamental_sequence}.  The first steps of the construction give
\begin{equation}
\label{equ:self-injective_small_ET}
  E^0 = 0
  \;\;,\;\;
  E^1 \cong P_0 \underset{ k }{ \otimes } B^0
  \;\;,\;\;
  T^0 \cong M_0 \underset{ k }{ \otimes } B^0.
\end{equation}

(ii)  Part (i) permits us to construct a short exact sequence of inverse systems as follows:  Set
\[
  \Delta^i
  =
  \left\{
    \begin{array}{cl}
      \id_{ T^0 } & \mbox{for $i=0$,} \\[1mm]
      \delta^1 \circ \cdots \circ \delta^i & \mbox{for $i \geqslant 1$}
    \end{array}
  \right.
\]  
and
\[
  W^i = \Ker \Delta^i.
\]
Each $\Delta^i$ is an epimorphism because each $\delta^i$ is an epimorphism by Diagram \eqref{equ:pushout}.  Hence there is a short exact sequence of inverse systems, where it is easy to check that the induced morphisms $\omega^i$ are also epimorphisms:
\begin{equation}
\label{equ:25_11}
\vcenter{
  \xymatrix @-0.98pc {
    \cdots \ar[rr] && W^2 \ar@{->>}^{ \omega^2 }[rr] \ar@{^(->}[dd] && W^1 \ar@{->>}^{ \omega^1 }[rr] \ar@{^(->}[dd] && W^0 \ar@{^(->}[dd] \\
    \\
    \cdots \ar[rr] && T^2 \ar@{->>}^{ \delta^2 }[rr] \ar@{->>}_{ \Delta^2 }[dd] && T^1 \ar@{->>}^{ \delta^1 }[rr] \ar@{->>}^{ \Delta^1 }[dd] && T^0 \ar@{->>}^{ \Delta^0 }[dd] \\
    \\
    \cdots \ar[rr] && T^0 \ar@{=}[rr] && T^0 \ar@{=}[rr] && T^0 \lefteqn{.} \\
            }
        }
\end{equation}
The inverse limits of the two first systems will be denoted
\begin{equation}
\label{equ:WT}
  W = \invlim\, W^i
  \;\;,\;\;
  T = \invlim\, T^i.
\end{equation}
The inverse limit of the third system is
\begin{equation}
\label{equ:T0}
  \invlim\, T^0
  \cong
  \invlim\, M_0 \underset{ k }{ \otimes } B^0
  \cong 
  M_0 \underset{ k }{ \otimes } B^0
\end{equation}
by Equation \eqref{equ:self-injective_small_ET}.  
\end{Construction}

\begin{Remark}
\label{rmk:25}
If $q \in Q$ and $B \in \cB$ are given, then we can set $M_0 = S \langle q \rangle$ and $B^0 = B$ in Setup \ref{set:self-injective2}.  We will prove that if $( \cA,\cB )$ is hereditary, then the inverse limit of \eqref{equ:25_11} is a short exact sequence
\begin{equation}
\label{equ:lim_sequence}
  0
  \xrightarrow{}
  W
  \xrightarrow{}
  T
  \xrightarrow{}
  S_q( B )
  \xrightarrow{}
  0,
\end{equation}
which can be used as the sequence in condition (Seq)(ii) of Definition \ref{def:ExSeq}.  This will be accomplished in Proposition \ref{pro:25}.  

As an example, suppose that $Q$ is the quiver with relations \eqref{equ:Gillespie_quiver}, viewed as a $k$-preadditive category.  Then $\cX$ is the category of chain complexes over ${}_{ R }\!\Mod$.  If $q = 1$ and we write $I^j = I_{ -j }$, then 
\[
  T^i = \cdots \xrightarrow{} 0 \xrightarrow{} B \xrightarrow{} I_0 \xrightarrow{} I_{ -1 } \xrightarrow{} \cdots \xrightarrow{} I_{ -i+1 } \xrightarrow{} 0 \xrightarrow{} \cdots
\]
with $B$ in degree $1$, and
\[
  W^i = \cdots \xrightarrow{} 0 \xrightarrow{} I_0 \xrightarrow{} I_{ -1 } \xrightarrow{} \cdots \xrightarrow{} I_{ -i+1 } \xrightarrow{} 0 \xrightarrow{} \cdots.
\]
The inverse limits become the augmented injective resolution
\[
  T = \cdots \xrightarrow{} 0 \xrightarrow{} B \xrightarrow{} I_0 \xrightarrow{} I_{ -1 } \xrightarrow{} I_{ -2 } \xrightarrow{} \cdots
\]
with $B$ in degree $1$, and the injective resolution
\[
  W = \cdots \xrightarrow{} 0 \xrightarrow{} I_0 \xrightarrow{} I_{ -1 } \xrightarrow{} I_{ -2 } \xrightarrow{} \cdots.
\]
With these $T$ and $W$, the short exact sequence \eqref{equ:lim_sequence} is dual to a sequence with projective objects, which was used indirectly by Gillespie in his proof of compatibility, see the proof of \cite[thm.\ 3.12]{G2}.
\end{Remark}

\begin{Lemma}
\label{lem:8}
If $i \geqslant 0$ then $E^i \in \cE$.  
\end{Lemma}

\begin{proof}
It follows from Definition \ref{def:Phi_Psi} that $\cE$ contains $E^0 = 0$ and is closed under extensions.  Diagram \eqref{equ:pullback} contains the short exact sequence $0 \rightarrow E^i \rightarrow E^{ i+1 } \rightarrow P_i \underset{ k }{ \otimes } B^i \rightarrow 0$, so it is enough to show $P_i \underset{ k }{ \otimes } B^i \in \cE$ for each $i \geqslant 0$.  However, for $q \in Q$ we have $L_1C_q( P_i \underset{ k }{ \otimes } B^i ) \cong \Tor^Q_1( \dual\!S \langle q \rangle,P_i ) \underset{ k }{ \otimes } B^i = (*)$ by Lemma \ref{lem:self-injective_standard_isomorphisms}(iii).  Since $P_i$ is projective, this is
$(*) = 0 \underset{ k }{ \otimes } B^i = 0$.  This shows $P_i \underset{ k }{ \otimes } B^i \in \cE_L = \cE$.  
\end{proof}

\begin{Lemma}
\label{lem:self-injective_useful_exact_sequence}
If $i \geqslant 0$ and $q \in Q$ then there is a short exact sequence
\[
  0
  \xrightarrow{}
  K_q( M_{ i+1 } \underset{ k }{ \otimes } B^i ) 
  \xrightarrow{ K_q( M_{ i+1 } \otimes \beta^i ) }
  K_q( M_{ i+1 } \underset{ k }{ \otimes } I^i ) 
  \xrightarrow{ K_q( M_{ i+1 } \otimes \alpha^i ) }
  K_q( M_{ i+1 } \underset{ k }{ \otimes } B^{ i+1 } )
  \xrightarrow{}
  0.
\]
\end{Lemma}

\begin{proof}
By Paragraph \ref{bfhpg:tensor_and_hom}(iv) the functor $\Hom_Q( S \langle q \rangle,M_{ i+1 } ) \underset{ k }{ \otimes } -$ is exact.  Applying it to the short exact sequence $0 \xrightarrow{} B^i \xrightarrow{ \beta^i } I^i \xrightarrow{ \alpha^i }
B^{ i+1 } \xrightarrow{} 0$ gives the sequence in the lemma by Lemma \ref{lem:self-injective_standard_isomorphisms}(iv). 
\end{proof}

\begin{Lemma}
\label{lem:14}
If $i \geqslant 0$ and $q \in Q$ then:
\begin{enumerate}
\setlength\itemsep{4pt}

  \item  There is a short exact sequence
\[
  0 \xrightarrow{}
  K_q( E^i ) \xrightarrow{ K_q( \varepsilon^i ) }
  K_q( T^i ) \xrightarrow{ K_q( \tau^i ) }
  K_q( M_i \underset{ k }{ \otimes } B^i ) \xrightarrow{}
  0.
\]

  \item  There is an isomorphism
\[
  R^1K_q( T^i )
  \xrightarrow{ R^1K_q( \tau^i ) }
  R^1K_q( M_i \underset{ k }{ \otimes } B^i ).
\]

\end{enumerate}
\end{Lemma}

\begin{proof}
The functor $K_q( - ) = \Hom_Q( S \langle q \rangle,- )$ is left exact, so applying it to the short exact sequence \eqref{equ:self-injective_fundamental_sequence} gives a long exact sequence
\begin{align*}
  0
  & \xrightarrow{}
  K_q( E^i ) \xrightarrow{ K_q( \varepsilon^i ) }
  K_q( T^i ) \xrightarrow{ K_q( \tau^i ) }
  K_q( M_i \underset{ k }{ \otimes } B^i ) \\[1mm]
  & \xrightarrow{}
  R^1K_q( E^i ) \xrightarrow{ R^1K_q( \varepsilon^i ) }
  R^1K_q( T^i ) \xrightarrow{ R^1K_q( \tau^i ) }
  R^1K_q( M_i \underset{ k }{ \otimes } B^i ) 
  \xrightarrow{} R^2K_q( E^i ) \xrightarrow{} \cdots.
\end{align*}
This implies both parts of the lemma because $R^{ \geqslant 1 }K_q( E^i ) = 0$ by Proposition \ref{pro:self-injective_Ex} and Lemma \ref{lem:8}.
\end{proof}

\begin{Lemma}
\label{lem:11}
If $i \geqslant 1$ and $q \in Q$ then there is an exact sequence
\[
  K_q( E^{ i+1 } )
  \xrightarrow{ K_q( \kappa^{ i+1 } ) }
  K_q( T^i ) 
  \xrightarrow{ K_q( \tau^i ) }
  K_q( M_i \underset{ k }{ \otimes } B^i ).
\]
\end{Lemma}

\begin{proof}
Since $\Hom_Q( S \langle q \rangle,- )$ is left-exact and $\mu_i$ a monomorphism, $\Hom_Q( S \langle q \rangle,\mu_i )$ is injective.  But Setup \ref{set:self-injective2} implies
\[
  0 =
  \Hom_Q( S \langle q \rangle,\partial^P_i ) =
  \Hom_Q( S \langle q \rangle,\mu_i ) \circ \Hom_Q( S \langle q \rangle,\pi_i ),
\]
so we conclude $\Hom_Q( S \langle q \rangle,\pi_i ) = 0$.  By Lemma \ref{lem:self-injective_standard_isomorphisms}(iv) this implies
\begin{equation}
\label{equ:KqpiBi}
  K_q( \pi_i \underset{ k }{ \otimes } B^i ) =
  \Hom_Q( S \langle q \rangle,\pi_i ) \underset{ k }{ \otimes } B^i =  
  0.
\end{equation}
Now observe that the left exact functor $K_q$ preserves the pullback in Diagram \eqref{equ:pullback}, so there is the following pullback square.
\[
\vcenter{
  \xymatrix @-0.4pc {
    K_q( E^{ i+1 } ) \ar^{ K_q( \kappa^{ i+1 } ) }[rr] \ar_{ K_q( \eta^i ) }[dd] && K_q( T^i ) \ar^{ K_q( \tau^i ) }[dd] \\
    \\
    K_q( P_i \underset{ k }{ \otimes } B^i ) \ar_-{ K_q( \pi_i \otimes B^i ) }[rr] && K_q( M_i \underset{ k }{ \otimes } B^i ) \\
            }
        }
\]
Combining with Equation \eqref{equ:KqpiBi} proves the lemma. 
\end{proof}

\begin{Lemma}
\label{lem:12}
If $i \geqslant 1$ and $q \in Q$ then $\Image K_q( \kappa^i ) = \Image K_q( \delta^i )$. 
\end{Lemma}

\begin{proof}
If $i \geqslant 0$, then $K_q$ can be applied to Diagram \eqref{equ:pushout}.  Replacing the third non-trivial column by the images of the relevant morphisms gives the following commutative diagram. 
\[
\vcenter{
  \xymatrix @-0.98pc {
    && 0 \ar[dd] && 0 \ar[dd] && 0 \ar[dd] \\
    \\
    0 \ar[rr] && K_q( M_{ i+1 } \underset{ k }{ \otimes } B^i ) \ar^-{ K_q( \theta^{ i+1 } ) }[rr] \ar_{ K_q( M_{ i+1 } \otimes \beta^i ) }[dd] && K_q( E^{ i+1 } ) \ar[rr]\ar^{ K_q( \varepsilon^{ i+1 } ) }[dd] && \Image K_q( \kappa^{ i+1 } ) \ar[rr] \ar[dd] && 0 \\
    \\
    0 \ar[rr] && K_q( M_{ i+1 } \underset{ k }{ \otimes } I^i ) \ar_-{ K_q( \gamma^{ i+1 } ) }[rr] \ar_{ K_q( M_{ i+1 } \otimes \alpha^i ) }[dd] && K_q( T^{ i+1 } ) \ar[rr] \ar^{ K_q( \tau^{ i+1 } ) }[dd] && \Image K_q( \delta^{ i+1 } ) \ar[rr] \ar[dd] && 0 \\
    \\
    0 \ar[rr] && K_q( M_{ i+1 } \underset{ k }{ \otimes } B^{ i+1 } ) \ar@{=}[rr] \ar[dd] && K_q( M_{ i+1 } \underset{ k }{ \otimes } B^{ i+1 } ) \ar[rr] \ar[dd] && 0 \ar[rr] \ar[dd] && 0 \\
    \\
    && 0 && 0 && 0 \\
            }
        }
\]
It is enough to show that the third non-trivial column is a short exact sequence.  We use the Nine Lemma, so have to show that the rows and the first two non-trivial columns are short exact.  The row which contains an identity morphism is trivially short exact.  Since $K_q$ is left-exact, the other rows are short exact by construction.  The first non-trivial column is short exact by Lemma \ref{lem:self-injective_useful_exact_sequence} and the second by Lemma \ref{lem:14}(i).
\end{proof}

\begin{Lemma}
\label{lem:15}
If $i \geqslant 1$ and $q \in Q$ then $\Image K_q( \varepsilon^i ) = \Image K_q( \delta^{ i+1 } )$. 
\end{Lemma}

\begin{proof}
Using Lemmas \ref{lem:14}(i), \ref{lem:11}, and \ref{lem:12} gives the equalities
\[
  \Image K_q( \varepsilon^i ) =
  \Ker K_q( \tau^i ) =
  \Image K_q( \kappa^{ i+1 } ) =
  \Image K_q( \delta^{ i+1 } ).
\qedhere
\]
\end{proof}

\begin{Lemma}
\label{lem:21}
If $i \geqslant 1$ and $q \in Q$ then there are short exact sequences:
\begin{enumerate}
\setlength\itemsep{4pt}

  \item  $0 \xrightarrow{} \Image K_q( \delta^{ i+1 } ) \xrightarrow{ \psi' } K_q( T^i ) \xrightarrow{ K_q( \tau^i ) } K_q( M_i \underset{ k }{ \otimes } B^i ) \xrightarrow{} 0$,

  \item  $0 \xrightarrow{} K_q( M_{ i+1 } \underset{ k }{ \otimes } I^i ) \xrightarrow{ K_q( \gamma^{ i+1 } ) } K_q( T^{ i+1 } )\xrightarrow{ \psi'' }
\Image K_q( \delta^{ i+1 } ) \xrightarrow{} 0$,
  
\end{enumerate}
where $\psi'$ is the canonical inclusion and $\psi''$ is induced by $K_q( \delta^{ i+1 } )$.
\end{Lemma}

\begin{proof}
Applying the left exact functor $K_q$ to the short exact sequence \eqref{equ:self-injective_fundamental_sequence} gives a long exact sequence containing
\[
  K_q( E^i )
  \xrightarrow{ K_q( \varepsilon^i ) }
  K_q( T^i )
  \xrightarrow{ K_q( \tau^i ) }
  K_q( M_i \underset{ k }{ \otimes } B^i )
  \xrightarrow{}
  R^1 K_q( E^i ).
\]
The last term is zero by Proposition \ref{pro:self-injective_Ex} and Lemma \ref{lem:8}, and $\Image K_q( \varepsilon^i ) = \Image K_q( \delta^{ i+1 } )$ by Lemma \ref{lem:15}, so we get the sequence (i).

Diagram \eqref{equ:pushout} contains the short exact sequence $0 \xrightarrow{} M_{ i+1 } \underset{ k }{ \otimes } I^i \xrightarrow{ \gamma^{ i+1 } } T^{ i+1 } \xrightarrow{ \delta^{ i+1 } } T^i \xrightarrow{} 0$.  Applying the left exact functor $K_q$ gives the sequence (ii).
\end{proof}

\begin{Definition}
\label{def:self-injective_varphi}
If $i \geqslant 0$ and $q \in Q$, then $\varphi^{ i+2 }$ is the unique morphism which makes the following square commutative, where the vertical morphisms are the canonical inclusions.
\[
\vcenter{
  \xymatrix @-0.98pc {
    \Image K_q( \delta^{ i+2 } ) \ar^{ \varphi^{ i+2 } }[rr] \ar@{^(->}[dd] && \Image K_q( \delta^{ i+1 } ) \ar@{^(->}[dd] \\
    \\
    K_q( T^{ i+1 } ) \ar_-{ K_q( \delta^{ i+1 } ) }[rr] && K_q( T^i ) \\
                     }
        }
\]
\end{Definition}

\begin{Lemma}
\label{lem:22}
If $i \geqslant 1$ and $q \in Q$ then there is a short exact sequence
\[
  0
  \xrightarrow{}
  K_q( M_{ i+1 } ) \underset{ k }{ \otimes } B^i
  \xrightarrow{}
  \Image K_q( \delta^{ i+2 } )
  \xrightarrow{ \varphi^{ i+2 } }
  \Image K_q( \delta^{ i+1 } )
  \xrightarrow{}
  0.
\]
\end{Lemma}

\begin{proof}
In view of Lemma \ref{lem:self-injective_standard_isomorphisms}(iv), it is enough to show that there is a commutative diagram as follows, in which the first non-trivial row is a short exact sequence.
\[
\vcenter{
  \xymatrix @-0.98pc {
    && 0 \ar[dd] && 0 \ar[dd] && 0 \ar[dd] \\
    \\
    0 \ar[rr] && K_q( M_{ i+1 } \underset{ k }{ \otimes } B^i ) \ar^-{ \psi''' }[rr] \ar_{ K_q( M_{ i+1 } \otimes \beta^i ) }[dd] && \Image K_q( \delta^{ i+2 } ) \ar^{ \varphi^{ i+2 } }[rr] \ar^{ \psi' }[dd] && \Image K_q( \delta^{ i+1 } ) \ar[rr] \ar@{=}[dd] && 0 \\
    \\
    0 \ar[rr] && K_q( M_{ i+1 } \underset{ k }{ \otimes } I^i ) \ar^-{ K_q( \gamma^{ i+1 } ) }[rr] \ar_{ K_q( M_{ i+1 } \otimes \alpha^i ) }[dd] && K_q( T^{ i+1 } ) \ar^{ \psi'' }[rr] \ar^{ K_q( \tau^{ i+1 } ) }[dd] && \Image K_q( \delta^{ i+1 } ) \ar[rr] \ar[dd] && 0 \\
    \\
    0 \ar[rr] && K_q( M_{ i+1 } \underset{ k }{ \otimes } B^{ i+1 } ) \ar@{=}[rr] \ar[dd] && K_q( M_{ i+1 } \underset{ k }{ \otimes } B^{ i+1 } ) \ar[rr] \ar[dd] && 0 \ar[rr] \ar[dd] && 0 \\
    \\
    && 0 && 0 && 0 \\
            }
        }
\]
To construct the diagram, observe that it has three non-trivial columns, each of which is short exact.  The first comes from Lemma \ref{lem:self-injective_useful_exact_sequence}, the second from Lemma \ref{lem:21}(i), and the third is trivial.  As for the morphisms between the columns, $K_q( \gamma^{ i+1 } )$ is obtained from Diagram \eqref{equ:pushout}, which also shows that the lower left square is commutative.  There is a unique induced morphism $\psi'''$ making the upper left square commutative.  The morphism $\psi''$ is induced by $K_q( \delta^{ i+1 } )$, and the upper right square is commutative by Definition \ref{def:self-injective_varphi}.  The lower right square is trivially commutative.

We use the Nine Lemma, so it remains to show that the last two non-trivial rows are short exact.  The row which contains an identity morphism is trivially short exact, and the row above it is short exact by Lemma \ref{lem:21}(ii).
\end{proof}

\begin{Lemma}
\label{lem:19}
If $i \geqslant 1$ and $q \in Q$ then $R^1 K_q( \delta^i ) = 0$.  
\end{Lemma}

\begin{proof}
If $i \geqslant 0$ then Diagram \eqref{equ:pushout} contains a short exact sequence $0 \xrightarrow{} M_{ i+1 } \underset{ k }{ \otimes } I^i \xrightarrow{ \gamma^{ i+1 } } T^{ i+1 } \xrightarrow{ \delta^{ i+1} } T^i \xrightarrow{} 0$.  It induces a long exact sequence containing
\[
  R^1 K_q( M_{ i+1 } \underset{ k }{ \otimes } I^i )
  \xrightarrow{ R^1 K_q( \gamma^{ i+1 } ) } R^1 
  K_q( T^{ i+1 } )
  \xrightarrow{ R^1 K_q( \delta^{ i+1} ) } 
  R^1 K_q( T^i ),
\]
so it is enough to see that $R^1 K_q( \gamma^{ i+1 } )$ is an epimorphism.

Setup \ref{set:self-injective2} gives an epimorphism $I^i \xrightarrow{ \alpha^i } B^{ i+1 }$, and
\[
  R^1 K_q( M_{ i+1 } \underset{ k }{ \otimes } I^i ) 
  \xrightarrow{ R^1 K_q( M_{ i+1 } \otimes \alpha^i ) }
  R^1 K_q( M_{ i+1 } \underset{ k }{ \otimes } B^{ i+1 } )
\]
is an epimorphism because Lemma \ref{lem:self-injective_standard_isomorphisms}(iv) says it can be identified with the morphism obtained by applying the exact functor $\Ext_Q^1( S \langle q \rangle,M_{ i+1 } ) \underset{ k }{ \otimes } -$ to $\alpha^i$.  Combining this with Lemma \ref{lem:14}(ii) shows that applying $R^1 K_q$ to the lower square in Diagram \eqref{equ:pushout} gives the following commutative square with an epimorphism on the left and an isomorphism on the right. 
\[
\vcenter{
  \xymatrix @-0.98pc {
    R^1 K_q( M_{ i+1 } \underset{ k }{ \otimes } I^i ) \ar^-{ R^1 K_q( \gamma^{ i+1 } ) }[rr] \ar@{->>}_{ R^1 K_q( M_{ i+1 } \otimes \alpha^i ) }[dd] && R^1 K_q( T^{ i+1 } )  \ar^{ R^1 K_q( \tau^{ i+1 } ) }_{ \mbox{\rotatebox{90}{$\sim$}} }[dd] \\
    \\
    R^1 K_q( M_{ i+1 } \underset{ k }{ \otimes } B^{ i+1 } ) \ar@{=}[rr] && R^1 K_q( M_{ i+1 } \underset{ k }{ \otimes } B^{ i+1 } ) \\
            }
        }
\]
Hence $R^1 K_q( \gamma^{ i+1 } )$ is an epimorphism as desired.
\end{proof}

\begin{Lemma}
\label{lem:self-injective_limT}
For each $q \in Q$ we have:
\begin{enumerate}
\setlength\itemsep{4pt}

  \item  $K_q( T ) \cong \invlim\, K_q( T^i )$, 

  \item  There is a short exact sequence
\[
  0
  \xrightarrow{}
  R^1\invlim\, K_q( T^i )
  \xrightarrow{}
  R^1 K_q( T )
  \xrightarrow{}
  \invlim\, R^1 K_q( T^i )
  \xrightarrow{}
  0.
\]

\end{enumerate}
\end{Lemma}

\begin{proof}
Recall from Construction \ref{con:self-injective}(ii) that $T$ is the inverse limit of
$\cdots \xrightarrow{ \delta^2 } T^1 \xrightarrow{ \delta^1 } T^0$.  Each $\delta^i$ is an epimorphism by Diagram \eqref{equ:pushout}, so this system satisfies the Mittag-Leffler condition and there is a short exact sequence
\[
  0
  \xrightarrow{}
  T
  \xrightarrow{}
  \prod_i T^i
  \xrightarrow{ \id - \operatorname{ shift } }
  \prod_i T^i
  \xrightarrow{}
  0,
\]
see Paragraph \ref{bfhpg:invlim}.  It gives a long exact sequence containing
\begin{align}
\nonumber
  0
  & \xrightarrow{}
  K_q( T )
  \xrightarrow{}
  K_q\Big( \prod_i T^i \Big)
  \xrightarrow{ K_q( \id - \operatorname{ shift } ) }
  K_q\Big( \prod_i T^i \Big) \\[1mm]
\label{equ:20_1}
  & \xrightarrow{}
  R^1 K_q( T )
  \xrightarrow{}
  R^1 K_q\Big( \prod_i T^i \Big)
  \xrightarrow{ R^1 K_q( \id - \operatorname{ shift } ) }
  R^1 K_q\Big( \prod_i T^i \Big).
\end{align}
Combining Equation \eqref{equ:self-injective_CSK} and Paragraph \ref{bfhpg:Tor_and_Ext}(i) shows that there are natural isomorphisms
\[
  R^{ \ell } K_q\Big( \prod T^i \Big)
  \xrightarrow{ \cong }
  \prod R^{ \ell }K_q( T^i ),
\]
so \eqref{equ:20_1} can be identified with
\begin{align*}
  0
  & \xrightarrow{}
  K_q( T )
  \xrightarrow{}
  \prod_i K_q( T^i )
  \xrightarrow{ \id - \operatorname{ shift } }
  \prod_i K_q( T^i ) \\[1mm]
  & \xrightarrow{}
  R^1 K_q( T )
  \xrightarrow{}
  \prod_i R^1 K_q( T^i )
  \xrightarrow{ \id - \operatorname{ shift } }
  \prod R^1 K_q( T^i ),
\end{align*}
which implies both parts of the lemma.
\end{proof}

\begin{Lemma}
\label{lem:self-injective_KqMB}
Assume that $( \cA,\cB )$ is hereditary, that $B^0 \in \cB$, and that $M \in {}_{ Q }\!\Mod$ has finite length.  If $i \geqslant 0$ and $q \in Q$ then $K_q( M )\underset{ k }{ \otimes } B^i \in \cB$.  
\end{Lemma}

\begin{proof}
Since $( \cA,\cB )$ is hereditary, $B^i \in \cB$ for each $i \geqslant 0$ by \cite[lem.\ 5.24]{GT}.  Since $M \in {}_{ Q }\!\Mod$ has finite length, $\dim_k K_q( M ) < \infty$ by Paragraph \ref{bfhpg:simple}(v).  Hence $K_q( M ) \underset{ k }{ \otimes } B^i$ is in $\cB$ because it is a finite coproduct of copies of $B^i$.  
\end{proof}

\begin{Lemma}
\label{lem:self-injective_limU}
Assume that $( \cA,\cB )$ is hereditary and that $B^0 \in \cB$.  Then
\[
  \invlim
  \big(
  \cdots
  \xrightarrow{}
  \Image K_q( \delta^4 )
  \xrightarrow{ \varphi^4 }
  \Image K_q( \delta^3 )
  \xrightarrow{ \varphi^3 }
  \Image K_q( \delta^2 )
  \big)
\]
is in $\cB$.  
\end{Lemma}

\begin{proof}
Since $\cB = \cA^{ \perp }$, the Lukas Lemma implies that it is enough to show the following, see Paragraph \ref{bfhpg:invlim}.
\begin{enumerate}
\setlength\itemsep{4pt}

  \item  $\varphi^i$ is an epimorphism for $i \geqslant 3$.
  
  \item  $\Image K_q( \delta^2 ) \in \cB$.
  
  \item  $\Ker \varphi^i \in \cB$ for $i \geqslant 3$.  

\end{enumerate}
Lemma \ref{lem:22} gives (i).  It also gives $\Ker \varphi^{ i+2 } \cong K_q( M_{ i+1 } ) \underset{ k }{ \otimes } B^i$ for $i \geqslant 1$, and this is in $\cB$ by Lemma \ref{lem:self-injective_KqMB} since $M_{ i+1 }$ has finite length by Setup \ref{set:self-injective2}.  This shows (iii).  

To show (ii), we compute:
\[
  \Image K_q( \delta^2 )
  \stackrel{ \rm (a) }{ \cong }
  \Image K_q( \varepsilon^1 )
  \stackrel{ \rm (b) }{ \cong }
  K_q( E^1 )
  \stackrel{ \rm (c) }{ \cong }
  K_q( P_0 \underset{ k }{ \otimes } B^0 )
  \stackrel{ \rm (d) }{ \cong }
  K_q( P_0 ) \underset{ k }{ \otimes } B^0.
\]
Here (a) is by Lemma \ref{lem:15}.  For (b), apply the left exact functor $K_q$ to the short exact sequence \eqref{equ:self-injective_fundamental_sequence} for $i = 1$.  Equation \eqref{equ:self-injective_small_ET} implies (c), and Lemma \ref{lem:self-injective_standard_isomorphisms}(iv) implies (d).  But $K_q( P_0 ) \underset{ k }{ \otimes } B^0 \in \cB$ by 
Lemma \ref{lem:self-injective_KqMB} since $P_0$ has finite length by Setup \ref{set:self-injective2}.
\end{proof}

\begin{Lemma}
\label{lem:23}
Assume that $( \cA,\cB )$ is hereditary and that $B^0 \in \cB$.  Then $T \in \Psi( \cB )$.
\end{Lemma}

\begin{proof}
Let $q \in Q$ be given.  By Definition \ref{def:Phi_Psi} we must show $K_q( T ) \in \cB$ and $R^1K_q( T ) = 0$.

$K_q( T ) \in \cB$:  Lemma \ref{lem:21}(i) gives the vertical short exact sequences in the following diagram.
\[
\vcenter{
  \xymatrix @-0.98pc {
    \cdots \ar[rr] && \Image K_q( \delta^4 ) \ar^{ \varphi^4 }[rr] \ar@{^(->}[dd] && \Image K_q( \delta^3 ) \ar^{ \varphi^3 }[rr] \ar@{^(->}[dd] && \Image K_q( \delta^2 ) \ar@{^(->}[dd] \\
    \\
    \cdots \ar[rr] && K_q( T^3 ) \ar^{ K_q( \delta^3 ) }[rr] \ar@{->>}_{ K_q( \tau^3 ) }[dd] && K_q( T^2 ) \ar^{ K_q( \delta^2)  }[rr] \ar@{->>}^{ K_q( \tau^2 ) }[dd] && K_q( T^1 ) \ar@{->>}^{ K_q( \tau^1 ) }[dd] \\
    \\
    \cdots \ar[rr] && K_q( M_3 \underset{ k }{ \otimes } B^3 ) \ar_0[rr] && K_q( M_2 \underset{ k }{ \otimes } B^2 ) \ar_0[rr] && K_q( M_1 \underset{ k }{ \otimes } B^1 ) \\
            }
        }
\]
The upper squares are commutative by Definition \ref{def:self-injective_varphi}, and the lower squares are obviously commutative, so the diagram constitutes a short exact sequence of inverse systems.  The long exact $\invlim$ sequence contains
\[
  0
  \xrightarrow{}
  \invlim\, \Image K_q( \delta^i )
  \xrightarrow{}
  \invlim\, K_q( T^i )
  \xrightarrow{}
  \invlim\, K_q( M_i \underset{ k }{ \otimes } B^i ),
\]
where the last term is zero because all morphisms in the third inverse system are zero.  This gives the first of the following isomorphisms,
\[
  \invlim\, \Image K_q( \delta^i )
  \cong
  \invlim\, K_q( T^i )
  \cong
  K_q( T ),
\]
and the second isomorphism is by Lemma \ref{lem:self-injective_limT}(i).  However, the left hand side is in $\cB$ by Lemma \ref{lem:self-injective_limU}.

$R^1K_q( T ) = 0$:  We show this without using the assumption $B^0 \in \cB$.  If $i \geqslant 2$ and $q \in Q$ then the epimorphism in the short exact sequence of Lemma \ref{lem:22} shows
\[
  \Image \big( K_q( \delta^i ) \circ K_q( \delta^{ i+1 } ) \big) = \Image K_q( \delta^i ).
\]
Hence the system
\[
  \cdots \xrightarrow{ K_q( \delta^2 ) } K_q( T^1 ) \xrightarrow{ K_q( \delta^1 ) } K_q( T^0 )
\]  
satisfies the Mittag-Leffler condition, so the first term of the exact sequence in Lemma \ref{lem:self-injective_limT}(ii) is zero, see Paragraph \ref{bfhpg:invlim}.  The last term is zero because Lemma \ref{lem:19} says that the morphisms vanish in the inverse system 
\[
  \cdots \xrightarrow{ R^1 K_q( \delta^2 ) } R^1 K_q( T^1 ) \xrightarrow{ R^1 K_q( \delta^1 ) } R^1 K_q( T^0 ).
\]  
Hence $R^1 K_q( T ) = 0$ as desired.
\end{proof}

\begin{Proposition}
\label{pro:25}
In the situation of Setup \ref{set:self-injective}, if $( \cA,\cB )$ is hereditary then condition (Seq) of Definition \ref{def:ExSeq} holds.
\end{Proposition}

\begin{proof}
We show condition (Seq)(ii).  Condition (Seq)(i) follows by a dual argument, parts of which are simplified by exactness of direct limits.

Let $q \in Q$ and $B \in \cB$ be given.  Set $M_0 = S \langle q \rangle$ and $B^0 = B$ in Setup \ref{set:self-injective2}.
Consider the short exact sequence of inverse systems \eqref{equ:25_11}.
The morphisms in the inverse systems are epimorphisms, so Paragraph \ref{bfhpg:invlim} says there is an induced short exact sequence
\[
  0
  \xrightarrow{}
  \invlim\, W^i
  \xrightarrow{}
  \invlim\, T^i
  \xrightarrow{}
  \invlim\, T^0
  \xrightarrow{}
  0,
\]
which by Equations \eqref{equ:WT} and \eqref{equ:T0} reads
\[
  0
  \xrightarrow{}
  W
  \xrightarrow{}
  T
  \xrightarrow{}
  S_q( B )
  \xrightarrow{}
  0,
\]
where we have used $M_0 \underset{ k }{ \otimes } B^0 = S \langle q \rangle \underset{ k }{ \otimes } B = S_q( B )$ by Equation \eqref{equ:self-injective_CSK}.  We claim this sequence can be used as the sequence in condition (Seq)(ii).  

We have $T \in \Psi( \cB )$ by Lemma \ref{lem:23}, so it remains to show $\Ext_{ Q,R }^2( \cE,W ) = 0$.  By Lemma \ref{lem:self-injective_E_perp} it is enough to show $W \in \cE^{ \perp }$.  The Lukas Lemma can be applied to the first inverse system in \eqref{equ:25_11} because $\omega^i$ is an epimorphism for $i \geqslant 1$.  Hence it is sufficient to show the following, see Paragraph \ref{bfhpg:invlim}.
\begin{enumerate}
\setlength\itemsep{4pt}

  \item  $W^0 \in \cE^{ \perp }$.
  
  \item  $\Ker \omega^i \in \cE^{ \perp }$ for $i \geqslant 1$.  

\end{enumerate}
But (i) is trivially true because $W^0 = \Ker \Delta^0 = \Ker \id_{ T^0 } = 0$.  For (ii), let $i \geqslant 1$ be given.  From Diagram \eqref{equ:25_11} it is easy to prove the first of the isomorphisms
\[
  \Ker \omega^i
  \cong \Ker \delta^i
  \cong M_i \underset{ k }{ \otimes } I^{ i-1 }
  = (*),
\]
and the second isomorphism is by Diagram \eqref{equ:pushout}.  But $M_i$ has finite length and $I^{ i-1 }$ is injective by Setup \ref{set:self-injective2}, so $(*) \in \cE^{ \perp }$ by Lemma \ref{lem:24}.
\end{proof}


\section{Proof of Theorem \ref{thm:main2}}
\label{sec:main2}

Section \ref{sec:main2} continues to use Setup \ref{set:self-injective}, except that:
\begin{itemize}

  \item  $Q$ is the quiver with relations \eqref{equ:intro_quiver_2}, viewed as a $k$-preadditive category; see Section \ref{subsec:main2}.

\end{itemize}
We think of objects of ${}_{ Q }\!\Mod$ and $\cX$ as quiver representations.  In particular, the value of $X$ at $q$ is denoted $X_q$, instead of $X( q )$ which would be used if we thought of $X$ as a functor.  Recall from Section \ref{subsec:main2} that each $X \in \cX$ has the form
\[
\xymatrix
{
  X_{ N-1 } \ar^-{ \partial^X_{ N-1 } }[r] & X_{ N-2 } \ar^-{ \partial^X_{ N-2 } }[r] & \cdots \ar[r] & X_1 \ar^-{ \partial^X_1 }[r] & X_0, \ar^{ \partial^X_0 }@/^1.5pc/[llll] \\
}
\]
where two consecutive morphisms compose to $0$.  For $0 \leqslant q \leqslant N-1$ there is a homology functor $\cX \xrightarrow{ \H_q } {}_{ R }\!\Mod$ defined in an obvious fashion.

\begin{Lemma}
\label{lem:main2_functors}
For $0 \leqslant q \leqslant N-1$ and $X \in \cX$ we have
\[
  C_q( X ) = \Coker( \partial^X_{ q+1 } )
  \;\;,\;\;
  K_q( X ) = \Ker( \partial^X_q )
  \;\;,\;\;
  R^1K_q = \H_{ q-1 }
  \;\;,\;\;
  L_1C_q = \H_{ q+1 },
\]
with subscripts taken modulo $N$.
\end{Lemma}

\begin{proof}
The functor $S_q : \cM \xrightarrow{} \cX$ sends an object $M$ to an object $S_q( M )$ which has $M$ at vertex $q$ and $0$ at all other vertices.  The two first formulae in the lemma are easily verified to define left and right adjoint functors to $S_q$, hence define $C_q$ and $K_q$.

The simple object $S \langle q \rangle$ has $k$ at vertex $q$ and $0$ at every other vertex.  There is an indecomposable projective object $P \langle q \rangle$, see Paragraph \ref{bfhpg:simple}(i).  It has copies of $k$ at vertices $q$ and $q-1$ and $0$ at every other vertex.  The homomorphism between the copies of $k$ is the identity map, and vertices are taken modulo $N$.  This permits to determine the minimal augmented projective resolution of $S \langle q \rangle$ in ${}_{ Q }\!\Mod$.  The first terms are the following, with indices taken modulo $N$.
\[
  \cdots
  \xrightarrow{}
  P \langle q-2 \rangle
  \xrightarrow{}
  P \langle q-1 \rangle
  \xrightarrow{}
  P \langle q \rangle
  \xrightarrow{}
  S \langle q \rangle
  \xrightarrow{}
  0
  \xrightarrow{}
  \cdots
\]
Each morphism of projective objects is induced by an arrow in $Q$.  We can now compute $R^1K_q( X ) = \Ext_Q^1( S \langle q \rangle,X )$ by using the projective resolution and Paragraph \ref{bfhpg:simple}(iii).  This gives the third formula in the lemma, and the fourth formula is proved similarly.
\end{proof}

\begin{proof}[Proof of Theorem \ref{thm:main2}]
Paragraph \ref{bfhpg:quivers} means that Theorems \ref{thm:self-injective_main} and \ref{thm:self-injective_complete} apply to the setup of Theorem \ref{thm:main2}.  The formulae for $\Phi( \cA )$ and $\Psi( \cB )$ in Theorem \ref{thm:self-injective_main} can be converted into the formulae in Theorem \ref{thm:main2}, part (i) by using Lemma \ref{lem:main2_functors}, and Theorem \ref{thm:self-injective_complete} implies Theorem \ref{thm:main2}, part (ii).
\end{proof}

\section{Proof of Theorem \ref{thm:main3}}
\label{sec:ZA_3}

Section \ref{sec:ZA_3} continues to use Setup \ref{set:self-injective}, except that:
\begin{itemize}

  \item  $Q$ is the repetitive quiver $\BZ A_3$ modulo the mesh relations, viewed as a $k$-preadditive category; see Section \ref{subsec:ZA_3}.

\end{itemize}
As in Section \ref{sec:main2} we think of objects of ${}_{ Q }\!\Mod$ and $\cX$ as quiver representations.  For $j \in \BZ$ there is an arrow $( j,0 ) \xrightarrow{} ( j,1 )$ in $Q$, so a corresponding homomorphism $X_{ (j,0) } \xrightarrow{} X_{ (j,1) }$ for each $X \in \cX$.  This and similar homomorphisms are used in the following two lemmas.

\begin{Lemma}
\label{lem:ZA_3_CjKj}
For $j \in \BZ$ and $X \in \cX$ we have:
\begin{enumerate}
\setlength\itemsep{4pt}

  \item The functors $C_q$ are given by
\begin{align*}
  C_{ (j,0) }( X ) & =
  \Coker( X_{ (j+1,1) } \xrightarrow{} X_{ (j,0) }  ), \\[1mm]
  C_{ (j,1) }( X ) & =
  \Coker( X_{ (j,0) } \oplus X_{ (j+1,2) } \xrightarrow{} X_{ (j,1) } ), \\[1mm]
  C_{ (j,2) }( X ) & =
  \Coker( X_{ (j,1) } \xrightarrow{} X_{ (j,2) } ).
\end{align*}

  \item  The functors $K_q$ are given by
\begin{align*}
  K_{ (j,0) }( X ) & =
  \Ker( X_{ (j,0) } \xrightarrow{} X_{ (j,1) } ), \\[1mm]
  K_{ (j,1) }( X ) & =
  \Ker( X_{ (j,1) } \xrightarrow{} X_{ (j-1,0) } \oplus X_{ (j,2) } ), \\[1mm]
  K_{ (j,2) }( X ) & =
  \Ker( X_{ (j,2) } \xrightarrow{} X_{ (j-1,1) } ).
\end{align*}

\end{enumerate}
\end{Lemma}

\begin{proof}
The functor $S_q : \cM \xrightarrow{} \cX$ sends an object $M$ to an object $S_q( M )$ which has $M$ at vertex $q$ and $0$ at all the other vertices.  The formulae in the lemma are easily verified to define left and right adjoint functors to $S_q$, hence define $C_q$ and $K_q$.
\end{proof}

\begin{Lemma}
\label{lem:ZA_3_RiKj}
For $j \in \BZ$ and $X \in \cX$ we have:
\begin{align*}
  R^1K_{ (j,0) }( X ) & =
  \H( X_{ (j,0) } \xrightarrow{} X_{ (j,1) } \xrightarrow{} X_{ (j-1,0) } ), \\[1mm]
  R^1K_{ (j,1) }( X ) & =
  \H( X_{ (j,1) } \xrightarrow{} X_{ (j-1,0) } \oplus X_{ (j,2) } \xrightarrow{} X_{ (j-1,1) } ), \\[1mm]
  R^1K_{ (j,2) }( X ) & =
  \H( X_{ (j,2) } \xrightarrow{} X_{ (j-1,1) } \xrightarrow{} X_{ (j-1,2) } ).
\end{align*}
Here $\H$ denotes the homology of a three term chain complex, taken at the middle term.  The mesh relations imply that the arguments of $\H$ are indeed chain complexes.
\end{Lemma}

\begin{proof}
For readability, the simple object $S \langle (j,\ell) \rangle$ of ${}_{ Q }\!\Mod$ is denoted $S \langle j,\ell \rangle$.  It has $k$ at vertex $( j,\ell )$ and $0$ at every other vertex.  The indecomposable projective object $P \langle (j,\ell) \rangle$ of ${}_{ Q }\!\Mod$ is denoted $P \langle j,\ell \rangle$.  It is one of the following, where in each case, one of the vertices is identified by a superscript.
\begin{align*}
  P \langle j,0 \rangle
  & =
    \vcenter{
  \xymatrix @+0.5pc @!0 {
    & 0 \ar[dr] && 0 \ar[dr] && k \ar[dr] && 0 \ar[dr] && 0 \\
    \cdots && 0 \ar[dr] \ar[ur] && k \ar[dr] \ar^{ \id }[ur] && 0 \ar[dr] \ar[ur] && 0 \ar[dr] \ar[ur] && \cdots \\
    & 0 \ar[ur] && k^{ (j,0) } \ar^{ \id }[ur] && 0 \ar[ur] && 0 \ar[ur] && 0 \\
                        }
           } \\[1mm]
  P \langle j,1 \rangle
  & =
    \vcenter{
  \xymatrix @+0.5pc @!0 {
    & 0 \ar[dr] && 0 \ar[dr] && k \ar^{ \id }[dr] && 0 \ar[dr] && 0 \\
    \cdots && 0 \ar[dr] \ar[ur] && k^{ (j,1) } \ar_{ \id }[dr] \ar^{ \id }[ur] && k \ar[dr] \ar[ur] && 0 \ar[dr] \ar[ur] && \cdots \\
    & 0 \ar[ur] && 0 \ar[ur] && k \ar_{ -\id }[ur] && 0 \ar[ur] && 0 \\
                        }
           } \\[1mm]
  P \langle j,2 \rangle
  & =
    \vcenter{
  \xymatrix @+0.5pc @!0 {
    & 0 \ar[dr] && 0 \ar[dr] && k^{ (j,2) } \ar^{ \id }[dr] && 0 \ar[dr] && 0 \\
    \cdots && 0 \ar[dr] \ar[ur] && 0 \ar[dr] \ar[ur] && k \ar^{ \id }[dr] \ar[ur] && 0 \ar[dr] \ar[ur] && \cdots \\
    & 0 \ar[ur] && 0 \ar[ur] && 0 \ar[ur] && k \ar[ur] && 0 \\
                        }
           }           
\end{align*}
This permits to determine the minimal augmented projective resolutions of the simple objects in ${}_{ Q }\!\Mod$.  In each case, the first terms are the following.
\begin{align*}
  \cdots \xrightarrow{} P \langle j-1,0 \rangle \xrightarrow{} P \langle j,1 \rangle \xrightarrow{} P \langle j,0 \rangle & \xrightarrow{} S \langle j,0 \rangle \xrightarrow{} 0 \xrightarrow{} \cdots \\[1mm]
  \cdots \xrightarrow{} P \langle j-1,1 \rangle \xrightarrow{} P \langle j-1,0 \rangle \oplus P \langle j,2 \rangle \xrightarrow{} P \langle j,1 \rangle & \xrightarrow{} S \langle j,1 \rangle \xrightarrow{} 0 \xrightarrow{} \cdots \\[1mm]
  \cdots \xrightarrow{} P \langle j-1,2 \rangle \xrightarrow{} P \langle j-1,1 \rangle \xrightarrow{} P \langle j,2 \rangle & \xrightarrow{} S \langle j,2 \rangle \xrightarrow{} 0 \xrightarrow{} \cdots
\end{align*}
Each morphism of projective objects is induced by arrows in $Q$.  We can now compute $R^1K_{ (j,\ell) }( X ) = \Ext_Q^1( S \langle j,\ell \rangle,X )$ by using the projective resolutions and Paragraph \ref{bfhpg:simple}(iii), and this gives the formulae in the lemma.
\end{proof}

\begin{proof}[Proof of Theorem \ref{thm:main3}]
Paragraph \ref{bfhpg:quivers} means that Theorem \ref{thm:self-injective_main} applies to the setup of Theorem \ref{thm:main3}.  The formulae for $\Phi( \cA )$ and $\Psi( \cB )$ in Theorem \ref{thm:self-injective_main} can be converted into the formulae in Theorem \ref{thm:main3} by combining Definition \ref{def:Phi_Psi}, Proposition \ref{pro:self-injective_Ex}, and Lemmas \ref{lem:ZA_3_CjKj} and \ref{lem:ZA_3_RiKj}.
\end{proof}

\appendix

%
\numberwithin{equation}{section}
\renewcommand{\theequation}{\Alph{section}.\arabic{equation}}

\renewcommand{\thesubsection}{\Alph{section}.\Roman{subsection}}

\section{Compendium on functor categories}
\label{app:functor}

In this appendix, $k$, $Q$, and $R$ are as in Setup \ref{set:self-injective}:  $k$ is a field, $R$ a $k$-algebra, $Q$ a small $k$-preadditive category satisfying conditions (Fin), (Rad), and (SelfInj) of Definition \ref{def:FinRadSelfInj}.
The homomorphism functor and the radical of $Q$ will be denoted $Q( -,- )$ and $\rad_Q( -,- )$, see \cite[sec.\ A.3]{ASS}, \cite[sec.\ 3.2]{GabRoi}, and \cite[p.\ 303]{Kelly}.

The appendix explains some properties of the functor categories ${}_{ Q }\!\Mod$, $\Mod_Q$, ${}_{ Q,R }\!\Mod$, and ${}_{ Q }\!\Mod_Q$ from Paragraph \ref{bfhpg:functor_categories}, which are used extensively in Sections \ref{sec:self-injective} through \ref{sec:Seq}.  They share many properties of the module categories 
${}_{ \Lambda }\!\Mod$, $\Mod_{ \Lambda }$, ${}_{ \Lambda,R }\!\Mod$, and ${}_{ \Lambda }\!\Mod_{ \Lambda }$, where $\Lambda$ is a finite dimensional $k$-algebra.  This follows from conditions (Fin) and (Rad), which imply that $Q$ is a locally bounded spectroid in the terminology of \cite[secs.\ 3.5 and 8.3]{GabRoi}.  We can even think of $\Lambda$ as self-injective because condition (SelfInj) implies that projective, injective, and flat objects coincide in each of ${}_{ Q }\!\Mod$ and $\Mod_Q$, see Paragraph \ref{bfhpg:flat_and_injective}.
  
Note that each statement in the appendix for ${}_{ Q }\!\Mod$ has an analogue for $\Mod_Q$.

The appendix has been included because we do not have references for all the results we need on functor categories.  Some hold by \cite{GabRoi} as we shall point out along the way.  The rest follow by amending the proofs in the following references: \cite[chp.\ 1]{AusRepDim}, \cite[secs.\ 1-4]{AusRepI}, \cite[secs.\ 1 and 2]{AusRepII}, \cite[app.\ B]{JL}, \cite{OR}.

\begin{bfhpg}[\bf Hom and tensor functors]
\label{bfhpg:tensor_and_hom}
The following functors are used extensively in this paper.
\begin{enumerate}
\setlength\itemsep{4pt}

  \item  The homomorphism functor of ${}_{ Q }\!\Mod$ is
\[
  \Hom_{ Q }( -,- )
    : ({}_{ Q }\!\Mod)^{ \opp } \times {}_{ Q }\!\Mod \rightarrow 
        {}_{ k }\!\Mod.
\]
It is defined by $\Hom_Q( M,N )$ being the set of $k$-linear natural transformations $M \rightarrow N$ for $M,N \in {}_{ Q }\!\Mod$.    

\smallskip
\noindent
If $X \in {}_{ Q,R }\!\Mod$, then $X$ is a $k$-linear functor $Q \rightarrow {}_{ R }\!\Mod$, so $R$ acts on $\Hom_Q( M,X )$.  Hence $\Hom_Q$ can also be viewed as a functor
\[
  \Hom_{ Q }( -,- )
    : ({}_{ Q }\!\Mod)^{ \opp } \times {}_{ Q,R }\!\Mod \rightarrow 
        {}_{ R }\!\Mod.
\]

  \item  There are functors
\begin{align*}
  - \underset{Q}{\otimes} -
    & : \Mod_{ Q } \times \: {}_{ Q }\!\Mod \rightarrow {}_{ k }\!\Mod, \\[1mm]
  - \underset{Q}{\otimes} -
    & : \Mod_{ Q } \times \: {}_{ Q,R }\!\Mod \rightarrow {}_{ R }\!\Mod, \\[1mm]
  - \underset{Q}{\otimes} -
    & : {}_{ Q }\!\Mod_{ Q } \times \: {}_{ Q,R }\!\Mod \rightarrow {}_{ Q,R }\!\Mod,
\end{align*}
see \cite[p.\ 93]{OR}.  They are right exact in each variable, and the last of them satisfies
\begin{equation}
\label{equ:self-injective_identity_bimodule}
  Q( -,- ) \underset{ Q }{ \otimes } X \cong X
\end{equation}
naturally in $X \in {}_{ Q,R }\!\Mod$.  This makes sense because $Q( -,- )$ is an object of ${}_{ Q }\!\Mod_Q$.

  \item  There is a functor
\[
  \Hom_k( -,- )
    : (\Mod_{ Q })^{ \opp } \times {}_{ R }\!\Mod \rightarrow 
        {}_{ Q,R }\!\Mod
\]
defined by
\[
  \Hom_k( N,B )( q ) = \Hom_k\big( N( q ),B \big),
\]
where $\Hom_k$ on the right hand side is $\Hom$ of $k$-vector spaces.  It is exact in both variables.

  \item  There is a functor
\[
  - \underset{k}{\otimes} -
    : {}_{ Q }\!\Mod \times {}_{ R }\!\Mod \rightarrow 
        {}_{ Q,R }\!\Mod
\]
defined by
\[
  ( M \underset{k}{\otimes} B )( q )
  = M( q ) \underset{k}{\otimes} B,
\]
where $\underset{k}{\otimes}$ on the right hand side is tensor of $k$-vector spaces.  It is exact in both variables.

  \item  There is a functor
\[
  - \underset{k}{\otimes} -
    : \Mod_Q \times {}_{ Q }\!\Mod \rightarrow 
        {}_{ Q }\!\Mod_Q
\]
defined by
\[
  ( M \underset{k}{\otimes} N )( q',q'' )
  = M( q' ) \underset{k}{\otimes} N( q'' ),
\]
where $\underset{k}{\otimes}$ on the right hand side is tensor of $k$-vector spaces.  It is exact in both variables.  

  \item  We can view $\dual( - ) = \Hom_k( -,k )$ as a functor ${}_{ Q }\!\Mod \xrightarrow{ \dual } \Mod_Q$.  

\end{enumerate}
\end{bfhpg}

\begin{bfhpg}[\bf Standard isomorphisms]
\label{bfhpg:standard}
The functors from Paragraph \ref{bfhpg:tensor_and_hom} permit the following standard isomorphisms, among others.
\begin{enumerate}
\setlength\itemsep{4pt}

  \item  There is an adjunction isomorphism in ${}_{ k }\!\Mod$,
\[
  \Hom_{ Q,R } ( M \underset{k}{\otimes} B,X )
  \xrightarrow{ \cong }
  \Hom_R \big( B,\Hom_Q( M,X ) \big),
\]
natural in $M \in {}_{ Q }\!\Mod$, $B \in {}_{ R }\!\Mod$, $X \in {}_{ Q,R }\!\Mod$.

  \item  There is an adjunction isomorphism in ${}_{ k }\!\Mod$,
\[
  \Hom_R ( M \underset{Q}{\otimes} X,B )
  \xrightarrow{ \cong }
  \Hom_{ Q,R } \big( X,\Hom_k( M,B ) \big),
\]
natural in $M \in \Mod_Q$, $X \in {}_{ Q,R }\!\Mod$, $B \in {}_{ R }\!\Mod$.

  \item  There is an associativity isomorphism in ${}_{ R }\!\Mod$,
\[
  ( M \underset{Q}{\otimes} N ) \underset{k}{\otimes} B
  \xrightarrow{ \cong }
  M \underset{Q}{\otimes} ( N \underset{k}{\otimes} B ),
\]
natural in $M \in \Mod_Q$, $N \in {}_{ Q }\!\Mod$, $B \in {}_{ R }\!\Mod$, where $\underset{k}{\otimes}$ on the left hand side is tensor of $k$-vector spaces.

  \item  There is an associativity isomorphism in ${}_{ Q,R }\!\Mod$, 
\[
  ( M \underset{k}{\otimes} N ) \underset{Q}{\otimes} X
  \xrightarrow{ \cong }
  N \underset{k}{\otimes} ( M \underset{Q}{\otimes} X ),
\]
natural in $M \in \Mod_Q$, $N \in {}_{ Q }\!\Mod$, $X \in {}_{ Q,R }\!\Mod$.

  \item  There is a morphism in ${}_{ R }\!\Mod$,
\[
  \Hom_Q( M,N ) \underset{k}{\otimes} B
  \xrightarrow{}
  \Hom_Q( M,N \underset{k}{\otimes} B ),
\]
natural in $M,N \in {}_{ Q }\!\Mod$, $B \in {}_{ R }\!\Mod$.  It is an isomorphism if $M$ has finite length.  Note that $\underset{k}{\otimes}$ on the left hand side is tensor of $k$-vector spaces.

  \item  There is a morphism in ${}_{ Q,R }\!\Mod$,
\[
  M \underset{ k }{ \otimes } B
  \rightarrow
  \Hom_k ( \dual\!M,B ),
\]
natural in $M \in {}_{ Q }\!\Mod$, $B \in {}_{ R }\!\Mod$.  It is an isomorphism if $M$ has finite length. 

\end{enumerate}
\end{bfhpg}

\begin{bfhpg}[\bf Products and coproducts]
\label{bfhpg:product}
We will explain products and coproducts in ${}_{ Q }\!\Mod$.  What we say applies equally to $\Mod_Q$ and ${}_{ Q,R }\!\Mod$.

Let $\{ M_{ \alpha } \}$ be a family of objects of ${}_{ Q }\!\Mod$.  The product of the $M_{ \alpha }$ in ${}_{ Q }\!\Mod$ is given by
\[
  \Big( \prod_{ \alpha } M_{ \alpha } \Big)( - )
  = \prod_{ \alpha } M_{ \alpha }( - ),
\]
where the second $\prod$ is in ${}_{ k }\!\Mod$.  There is a similar formula for $\coprod$.  This implies that ${}_{ Q }\!\Mod$ inherits 
the following properties from ${}_{ k }\!\Mod$: It is complete and cocomplete, and products, coproducts, and filtered colimits preserve exact sequences.

Each of the tensor product functors from Paragraph \ref{bfhpg:tensor_and_hom} preserves coproducts in each variable.
\end{bfhpg}

\begin{bfhpg}[\bf Projective, injective, and simple objects]
\label{bfhpg:simple}
Each of the categories ${}_{ Q }\!\Mod$, $\Mod_Q$, ${}_{ Q,R }\!\Mod$, and ${}_{ Q }\!\Mod_Q$ has enough projective objects and enough injective objects.  We list some additional properties.
\begin{enumerate}
\setlength\itemsep{4pt}

  \item  By \cite[sec.\ 3.7]{GabRoi} we have the following:  For each $q \in Q$ there is an indecomposable projective object 
\[
  P \langle q \rangle = Q( q,- )
\]
in ${}_{ Q }\!\Mod$.  By conditions (Fin) and (Rad) of Definition \ref{def:FinRadSelfInj}, it has finite length and there is a unique maximal subobject $\fr P\langle q \rangle \subsetneq P \langle q \rangle$ given by $\fr P\langle q \rangle = \rad_Q( q,- )$.  The quotient
\[
  S \langle q \rangle
  = P \langle q \rangle / \fr P\langle q \rangle
\]
is a simple object in ${}_{ Q }\!\Mod$, which satisfies
\[
  S \langle q \rangle( p ) \cong 
  \left\{
    \begin{array}{cl}
      k & \mbox{ for } p=q, \\[1mm]
      0 & \mbox{ otherwise. }
    \end{array}
  \right.
\]
The simple objects of ${}_{ Q }\!\Mod$ are precisely the $S \langle q \rangle$ for $q \in Q$.  The simple objects of $\Mod_Q$ are precisely the duals $\dual\!S \langle q \rangle$ for $q \in Q$.  

  \item  By \cite[p.\ 85, exa.\ 2]{GabRoi} we have the following:  Each $M \in {}_{ Q }\!\Mod$ has an augmented projective resolution
\[
  \cdots
  \xrightarrow{ \partial_3 } P_2
  \xrightarrow{ \partial_2 } P_1
  \xrightarrow{ \partial_1 } P_0
  \xrightarrow{ \partial_0 } M
  \xrightarrow{} 0
  \xrightarrow{} \cdots,
\]
which can be constructed by choosing an epimorphism $P_0 \stackrel{ \partial_0 }{ \twoheadrightarrow } M$ with $P_0$ projective, then, when $\partial_{ i-1 }$ has been defined, choosing an epimorphism $P_i \twoheadrightarrow \Ker \partial_{ i-1 }$ and defining $\partial_i$ to be the composition $P_i \twoheadrightarrow \Ker \partial_{ i-1 } \hookrightarrow P_{ i-1 }$.  

\smallskip
\noindent
If $M$ has finite length, then condition (Fin) of Definition \ref{def:FinRadSelfInj} implies that each $P_i$ can be chosen as a coproduct of finitely many objects of the form $P \langle q \rangle$, and then each $P_i$ and each $\Ker \partial_i$ has finite length.
Moreover, by choosing each of the epimorphisms $P_0 \stackrel{ \partial_0 }{ \twoheadrightarrow } M$ and $P_i \twoheadrightarrow \Ker \partial_{ i-1 }$ as a projective cover, we can even make the resolution {\em minimal}, that is, if $i \geqslant 1$ then $\partial_i$ is in the radical of ${}_{ Q }\!\Mod$.  This implies that the functors $\dual\!S \langle q \rangle \underset{ Q }{ \otimes } -$ and $\Hom_Q( S \langle q \rangle,- )$ vanish on $\partial_i$.

  \item  A morphism $p \xrightarrow{ \pi } q$ in $Q$ induces a natural transformation $Q( q,- ) \rightarrow Q( p,- )$, that is, a morphism $P \langle q \rangle \rightarrow P \langle p \rangle$.  By Yoneda's Lemma, this in turn induces a commutative square
\[
\vcenter{
\xymatrix@-1.1pc
{
  \Hom_Q( P \langle p \rangle,X ) \ar[rr] \ar^{ \mbox{\rotatebox{90}{$\sim$}} }[dd] && \Hom_Q( P \langle q \rangle,X ) \ar[dd]^{ \mbox{\rotatebox{90}{$\sim$}} } \\
  \\
  X( p ) \ar_{ X( \pi ) }[rr] && X( q )\lefteqn{,}
}
        }
\]
natural in $X \in {}_{ Q,R }\!\Mod$, where the vertical arrows are isomorphisms.

  \item  If $X$ is a projective object of ${}_{ Q,R }\!\Mod$, then $X$ is projective when viewed as an object of ${}_{ Q }\!\Mod$.  

  \item  $M,N \in {}_{ Q }\!\Mod$ have finite length $\Rightarrow$ $\dim_k \Hom_Q( M,N ) < \infty$.

\end{enumerate}
\end{bfhpg}

\begin{bfhpg}[\bf $\Ext$ and $\Tor$ functors]
\label{bfhpg:Tor_and_Ext}
The functors $\Hom_Q$ and $\underset{ Q }{ \otimes }$ of Paragraph \ref{bfhpg:tensor_and_hom} have right and left derived functors,
\begin{align*}
  \Ext_{ Q }^i( -,- ) & :
    ({}_{ Q }\!\Mod)^{ \opp } \times {}_{ Q }\!\Mod 
    \rightarrow {}_{ k }\!\Mod, \\[1mm]
  \Tor^Q_i( - , - ) & :
    \Mod_{ Q } \times \: {}_{ Q }\!\Mod 
    \rightarrow {}_{ k }\!\Mod
\end{align*}
for $i \geqslant 0$.  Like $\Hom_Q$ and $\underset{ Q }{ \otimes }$ they can also be viewed as functors
\begin{align*}
  \Ext_{ Q }^i( -,- ) & :
  ({}_{ Q }\!\Mod)^{ \opp } \times {}_{ Q,R }\!\Mod 
  \rightarrow {}_{ R }\!\Mod, \\[1mm]
  \Tor^Q_i( -,- ) & :
  \Mod_{ Q } \times \: {}_{ Q,R }\!\Mod 
  \rightarrow {}_{ R }\!\Mod.
\end{align*}
We list some additional properties.
\begin{enumerate}
\setlength\itemsep{4pt}

  \item  Since products preserve exact sequences, there are isomorphisms in ${}_{ R }\!\Mod$,
\[
  \Ext_Q^i( N,\prod_{ \alpha }M_{ \alpha } )
  \xrightarrow{ \cong }
  \prod_{ \alpha } \Ext_Q^i( N,M_{ \alpha } ),
\]
natural in $N \in {}_{ Q }\!\Mod$ and $M_{ \alpha } \in {}_{ Q,R }\!\Mod$.  

  \item  The morphism in Paragraph \ref{bfhpg:standard}(v) induces standard morphisms in ${}_{ R }\!\Mod$,
\[
  \Ext_Q^i( M,N ) \underset{ k }{ \otimes } B
  \xrightarrow{}
  \Ext_Q^i( M,N \underset{ k }{ \otimes } B ),
\]
natural in $M,N \in {}_{ Q }\!\Mod$, $B \in {}_{ R }\!\Mod$.  They are isomorphisms if $M$ has finite length.

  \item  The isomorphism in Paragraph \ref{bfhpg:standard}(iii) induces standard isomorphisms in ${}_{ R }\!\Mod$,
\[
  \Tor^Q_i( M,N ) \underset{ k }{ \otimes } B
  \xrightarrow{ \cong }
  \Tor^Q_i( M,N \underset{ k }{ \otimes } B ),
\]
natural in $M \in \Mod_Q$, $N \in {}_{ Q }\!\Mod$, $B \in {}_{ R }\!\Mod$.

\end{enumerate}
\end{bfhpg}

\begin{bfhpg}[\bf Criteria for injectivity and flatness]
\label{bfhpg:flat_and_injective}
Condition (Fin) of Definition \ref{def:FinRadSelfInj} implies that $M \in {}_{ Q }\!\Mod$ satisfies
\begin{equation}
\label{equ:criterion_for_injectivity}
  \mbox{$M$ is injective $\;\Leftrightarrow\;$ If $q \in Q$ then $\Ext_Q^1( S \langle q \rangle,M ) = 0$.}
\end{equation}
Similarly, $M$ is {\em flat} if the functor $- \underset{ Q }{ \otimes } M$ is exact, and
\begin{equation}
\label{equ:criterion_for_flatness}
  \mbox{$M$ is flat $\;\Leftrightarrow\;$ If $q \in Q$ then $\Tor^Q_1( \dual\!S \langle q \rangle,M ) = 0$.}
\end{equation}
Conditions (Fin) and (SelfInj) of Definition \ref{def:FinRadSelfInj} imply that $M \in {}_{ Q }\!\Mod$ satisfies
\begin{equation}
\label{equ:flat_v_injective}
  \mbox{$M$ is projective $\;\Leftrightarrow\;$ $M$ is flat $\;\Leftrightarrow\;$ $M$ is injective.}
\end{equation}
\end{bfhpg}

\begin{bfhpg}[\bf Inverse limits]
\label{bfhpg:invlim}
We will explain inverse limits in ${}_{ Q }\!\Mod$.  What we say applies equally to $\Mod_Q$ and ${}_{ Q,R }\!\Mod$.  

Since products exist and preserve exact sequences, the results on (derived) inverse limits in \cite[sec.\ 3.5]{W} apply.  In particular, if there is an inverse system
\begin{equation}
\label{equ:appendix_inverse_system}
  \cdots \xrightarrow{} M^2 \xrightarrow{ \mu^2 } M^1 \xrightarrow{ \mu^1 } M^0,
\end{equation}
then there is an exact sequence
\[
  0
  \xrightarrow{} \invlim\, M^i
  \xrightarrow{} \prod_i M^i
  \xrightarrow{ \id - \operatorname{ shift } } \prod_i M^i
  \xrightarrow{} R^1\invlim\, M^i
  \xrightarrow{} 0,
\]
where $\id - \operatorname{ shift }$ is the difference between the identity morphism and the shift morphism induced by the $\mu^i$.

The inverse system is said to satisfy the {\em Mittag-Leffler condition} if, for each $i \geqslant 0$, the images of the maps $M^{ \ell } \rightarrow M^i$ for $\ell \geqslant i$ satisfy the descending chain condition.  In this case we have $R^1\invlim\, M^i = 0$.  This holds in particular if each morphism in \eqref{equ:appendix_inverse_system} is an epimorphism.

If there is a short exact sequence 
\[
\vcenter{
  \xymatrix @-0.98pc {
    \cdots \ar@{->>}[rr] && M'^2 \ar@{->>}[rr] \ar@{^(->}[dd] && M'^1 \ar@{->>}[rr] \ar@{^(->}[dd] && M'^0 \ar@{^(->}[dd] \\
    \\
    \cdots \ar[rr] && M^2 \ar[rr] \ar@{->>}[dd] && M^1 \ar[rr] \ar@{->>}[dd] && M^0 \ar@{->>}[dd] \\
    \\
    \cdots \ar[rr] && M''^2 \ar[rr] && M''^1 \ar[rr] && M''^0 \\
            }
        }
\]
of inverse systems, then there is an induced long exact sequence
\[
  0
  \xrightarrow{}
  \invlim\, M'_i
  \xrightarrow{}
  \invlim\, M_i
  \xrightarrow{}
  \invlim\, M''_i
  \xrightarrow{}
  R^1\invlim\, M'_i
  \xrightarrow{}
  R^1\invlim\, M_i
  \xrightarrow{}
  R^1\invlim\, M''_i
  \xrightarrow{} 0.
\]
If each morphism in the $M'$-system is an epimorphism, then $R^1\invlim\, M'_i = 0$ and there is a short exact sequence
\[
  0
  \xrightarrow{}
  \invlim\, M'_i
  \xrightarrow{}
  \invlim\, M_i
  \xrightarrow{}
  \invlim\, M''_i
  \xrightarrow{}
  0.
\]

Some of the results on inverse limits in \cite[sec.\ 6]{GT} also apply.  In particular, the Lukas Lemma says that if $N$ is fixed, then in order to conclude $\invlim\, M^i \in N^{ \perp }$, it is enough to verify the following for the inverse system \eqref{equ:appendix_inverse_system}, see \cite[lem.\ 6.37]{GT}.
\begin{enumerate}
\setlength\itemsep{4pt}

  \item  $\mu^i$ is an epimorphism for $i \geqslant 1$.
  
  \item  $M^0 \in N^{ \perp }$.
  
  \item  $\Ker \mu^i \in N^{ \perp }$ for $i \geqslant 1$.  

\end{enumerate}
\end{bfhpg}

\begin{bfhpg}[\bf The radical filtration]
\label{bfhpg:rad}
If $i \geqslant 0$ then the $i$'th power of the radical, $\rad_Q^i( -,- )$, is an object of ${}_{ Q }\!\Mod_Q$.  Because of condition (Rad) of Definition \ref{def:FinRadSelfInj}, there is a finite filtration in ${}_{ Q }\!\Mod_Q$, 
\begin{equation}
\label{equ:self-injective_radical_filtration}
  0 = \rad_Q^N \subsetneq \cdots \subsetneq \rad_Q^1 \subsetneq \rad_Q^0 = Q( -,- ),
\end{equation}
where $N \geqslant 0$ is chosen minimal such that $\rad_Q^N = 0$.  Each quotient $\rad_Q^i/\rad_Q^{ i-1 }$ is annihilated on both sides by $\rad_Q$, and this implies
\begin{equation}
\label{equ:self-injective_radical_quotients}
  \rad_Q^i/\rad_Q^{ i-1 }
  \cong
  \coprod_{ p,q \in Q }
    ( \dual\!S \langle p \rangle \underset{ k }{ \otimes } S \langle q \rangle )^{ n_i( p,q ) }
\end{equation}
for certain integers $n_i( p,q )$.  
\end{bfhpg}

\medskip
\noindent
{\bf Acknowledgement.}
We thank an anonymous referee for reading the paper carefully and making a number of useful suggestions.

We thank Jim Gillespie, Osamu Iyama, Berhard Keller, Sondre Kvamme, and Hiroyuki Minamoto for a number of illuminating comments, and Bernhard Keller for pointing out references \cite{GabRoi} and \cite{LP}.

This work was supported by EPSRC grant EP/P016014/1 ``Higher Dimensional Homological Algebra'' and LMS Scheme 4 Grant 41664.


\begin{thebibliography}{39}

\bibitem{ASS}  I.\ Assem, D.\ Simson, and A.\ Skowro\'{n}ski,
  ``Elements of the representation theory of associative algebras'',
  Vol.\ 1, London Math.\ Soc.\ Stud.\ Texts, Vol.\ 65, Cambridge
  University Press, Cambridge, 2006.

\bibitem{AusRepDim} M.\ Auslander, ``Representation dimension of Artin algebras'', Queen Mary College Ma\-the\-ma\-tics Notes, Queen Mary College, London, 1971.  Reprinted pp.\ 505--574 in: ``Selected works of Maurice Auslander'', Vol.\ 1 (edited by Reiten, Smal\o, and Solberg), American Mathematical Society, Providence, 1999.

\bibitem{AusRepI}  M.\ Auslander, {\it Representation theory of Artin algebras I}, Comm.\ Algebra {\bf 1} (1974), 177--268.

\bibitem{AusRepII}  M.\ Auslander, {\it Representation theory of Artin algebras II}, Comm.\ Algebra {\bf 1} (1974), 269--310.

\bibitem{BBE}  L.\ Bican, R.\ El Bashir, and E.\ E.\ Enochs, {\it All modules have flat covers}, Bull.\ London Math.\ Soc.\ {\bf 33} (2001), 385--390.

\bibitem{CE}  H.\ Cartan and S.\ Eilenberg, ``Homological algebra'', Princeton University Press, Princeton, 1956.  Reprinted in: Princeton
Landmarks Math., Princeton University Press, Princeton, 1999.

\bibitem{DY}  N.\ Ding and X.\ Yang, {\it 
On a question of Gillespie}, Forum Math.\ {\bf 27} (2015), 3205--3231. 





\bibitem{GabRoi}  P.\ Gabriel and A.\ V.\ Roiter, ``Algebra VIII: Representations of finite-dimensional algebras'', with a chapter by B.\ Keller, Encyclopaedia Math.\ Sci., Vol.\ 73, Springer-Verlag, Berlin, 1992. 


\bibitem{G0}  J.\ Gillespie, {\it Hereditary abelian model categories}, Bull.\ London Math.\ Soc.\ {\bf 48} (2016), 895--922.

\bibitem{G1}  J.\ Gillespie, {\it How to construct a Hovey triple from two cotorsion pairs}, Fund.\ Math.\ {\bf 230} (2015), 281--289.

\bibitem{G2}  J.\ Gillespie, {\it The flat model structure on {${\rm Ch}(R)$}}, Trans.\ Amer.\ Math.\ Soc.\ {\bf 356} (2004), 3369--3390.


\bibitem{GT}  R.\ G{\"o}bel and J.\ Trlifaj, ``Approximations and endomorphism algebras of modules'', 2nd revised and extended edition, de Gruyter Exp. Math., Vol.\ 41, Walter de Gruyter, Berlin/Boston, 2012.



\bibitem{Hirsch}  P.\ S.\ Hirschhorn, ``Model categories and their localizations'', Math.\ Surveys Monogr., Vol.\ 99, American Mathematical Society, Providence, 2003.


\bibitem{HJ2}  H.\ Holm and P.\ J\o rgensen, {\it Covers, precovers, and purity}, Illinois J.\ Math.\ {\bf 52} (2008), 691-703.

\bibitem{H}  M.\ Hovey, {\it Cotorsion pairs, model category structures, and representation theory}, Math.\ Z.\ {\bf 241} (2002), 553--592.

\bibitem{HBook}  M.\ Hovey, ``Model categories'', Math.\ Surveys Monogr., Vol.\ 63, American Mathematical Society, Providence, 1999.


\bibitem{IM1}  O.\ Iyama and H.\ Minamoto, $\cA$-derived categories, in preparation.

\bibitem{IM2}  O.\ Iyama and H.\ Minamoto, {\it On a generalization of complexes and their derived categories}, extended abstract of a talk at the 47th Ring and Representation Theory Symposium, Osaka City University, 2014.

\bibitem{JL}  C.\ U.\ Jensen and H.\ Lenzing, ``Model-theoretic algebra with particular emphasis on fields, rings, modules'', Algebra Logic Appl., Vol.\ 2, Gordon and Breach, New York, 1989.

\bibitem{K}  M.\ M.\ Kapranov, On the $q$-analog of homological algebra, preprint (1996).  {\tt arXiv:9611005v1. }


\bibitem{Kelly}  G.\ M.\ Kelly, {\it On the radical of a category}, J.\ Austral.\ Math.\ Soc.\ {\bf 4} (1964), 299--307.

\bibitem{LP}  B.\ Leclerc and P.-G.\ Plamondon, {\it Nakajima varieties and repetitive algebras}, Publ.\ Res.\ Inst.\ Math.\ Sci.\ {\bf 49} (2013), 531--561.

\bibitem{L}  J.\ Lurie, ``Higher topos theory'', Ann.\ of Math.\ Stud., Vol.\ 170, Princeton University Press, Princeton, 2009.


\bibitem{N}  H.\ Nakaoka, {\it A simultaneous generalization of mutation and recollement of cotorsion pairs on a triangulated category}, Appl.\ Categ.\ Structures {\bf 26} (2018), 491-544.

\bibitem{OR}  U.\ Oberst and H.\ Rohrl, {\it Flat and coherent functors}, J.\ Algebra {\bf 14} (1970), 91--105.



\bibitem{Q}  D.\ Quillen, ``Homotopical algebra'', Lecture Notes in Math., Vol.\ 43, Springer, Berlin--Heidelberg, 1967.

\bibitem{R}  J.\ J.\ Rotman, ``An introduction to homological algebra'', Second Edition, Springer, New York, 2009.

\bibitem{S}  L.\ Salce, {\it Cotorsion theories for abelian groups}, pp.\ 11-32 in: ``Symposia Mathematica, Vol.\ {XXIII} (Conf.\ Abelian
Groups and their Relationship to the Theory of Modules, INDAM, Rome, 1977)'', Academic Press, London--New York, 1979.


\bibitem{Sch}  J.\ Schr\"{o}er, {\it On the quiver with relations of a repetitive algebra}, Arch.\ Math.\ (Basel) {\bf 72} (1999), 426--432. 

\bibitem{Sto}  J.\ \v{S}\v{t}ov\'{\i}\v{c}ek, {\it Exact model categories, approximation theory, and cohomology of quasi-coherent sheaves}, pp.\ 297-367 in: ``Advances in representation theory of algebras'' (proceedings of the International Conference on Representations of Algebras, Bielefeld 2015, edited by Benson, Krause, and Skowro\'{n}ski), EMS Series of Congress Reports, European Mathematical Society, Z\"{u}rich, 2013.

\bibitem{W}  C.\ Weibel, ``An introduction to homological algebra'', Cambridge Stud.\ Adv.\ Math., Vol.\ 38, Cambridge University Press, Cambridge, 1994.

\end{thebibliography}
\end{document}